\documentclass{amsart}
\usepackage[utf8]{inputenc}    
\usepackage[T1]{fontenc}
\usepackage[english]{babel}     
\usepackage[a4paper, left=1in, right=1in, top=1in, bottom=1in]{geometry}
\usepackage{comment}
\usepackage{indentfirst}
\usepackage{graphicx}           
\usepackage{amsmath,amssymb}    
\usepackage{amsthm}             
\usepackage{mathtools}          
\usepackage{enumitem}           
\usepackage{listings}           
\usepackage{todonotes}          
\usepackage{hyperref}           
\usepackage[cal=boondox]{mathalfa}
\usepackage{relsize}
\usepackage{stmaryrd}
\usepackage{setspace}
\usepackage{chngcntr}

\counterwithin{equation}{subsection}
\counterwithin{figure}{subsection}

\setenumerate[1]{label={(\arabic*)}} 
\setlist[enumerate]{itemsep=0pt}
\setlist[itemize]{itemsep=0pt}

\newcommand\restr[2]{{
  \left.\kern-\nulldelimiterspace 
  #1 
  \right |_{#2} 
  }}

\newtheorem{theorem}{Theorem}[section]
\newtheorem{corollary}{Corollary}[theorem]
\newtheorem{lemma}[theorem]{Lemma}
\newtheorem{proposition}[theorem]{Proposition}
\newtheorem*{claim}{Claim}

\newtheorem{innercustomthm}{Theorem}
\newenvironment{customthm}[1]
  {\renewcommand\theinnercustomthm{#1}\innercustomthm}
  {\endinnercustomthm}

\theoremstyle{definition}
\newtheorem{definition}[theorem]{Definition}
\newtheorem{defprop}[theorem]{Definition/Proposition}

\newtheorem*{remark}{Remark(s)}

\begin{document}

\title{On dynamical coherence of partially hyperbolic flows}
\author{Mounib Abouanass}
\date{\today}

\maketitle

\begin{abstract}
    In this paper, we introduce the notion of dynamical coherence for a partially hyperbolic flow $(\varphi^t)$ on a smooth compact manifold $M$, and prove it under the assumption that there exists a compact foliation with trivial holonomy which integrates the subcenter distribution.
\end{abstract}

\tableofcontents

\section{Introduction}
The theory of dynamical systems, discrete or continuous, is particularly well-developed in the presence of hyperbolicity. In the discrete setting, a diffeomorphism $f$ on a smooth compact manifold $M$ is called \textit{Anosov} if the tangent bundle $TM$ splits as a direct sum $E^s \oplus E^u$ of two invariant, transverse subbundles, with $f$ contracting vectors in $E^s$ uniformly and expanding those in $E^u$ uniformly. For smooth flows, the definition is analogous, with both subbundles transverse to the flow direction.

In recent decades, interest has shifted toward a broader class of systems---\textit{partially hyperbolic} ones---pioneered notably in the works of Hirsch--Pugh--Shub \cite{hirsch_invariant_1977} and Brin--Pesin \cite{brin_partially_1974}. A diffeomorphism $f$ on $M$ is said to be (uniformly) partially hyperbolic if $TM$ admits a splitting into three invariant, transverse bundles $
TM = E^s \oplus E^c \oplus E^u,
$
where $E^s$ is uniformly contracted, $E^u$ uniformly expanded, and the \textit{center bundle} $E^c$ is neither contracted nor expanded as strongly. This class includes, for example, time-$t$ maps of Anosov flows. A smooth flow is called partially hyperbolic when its time-one map satisfies this definition; in this case, the flow direction is always contained in the center distribution.

In this article, we refine this definition for flows: we require that the center distribution further decomposes as
$E^c = \mathbb RX \oplus E^{\hat{c}}$,
where $X$ is the smooth non-vanishing vector field inducing the flow, and $E^{\hat{c}}$ is an invariant subbundle called the \textit{subcenter distribution}. This definition has the advantage of considering explicitly a subbundle transverse to the flow on which many techniques can be used (see \cite{fang_rigidity_2007}, \cite{abouanass_global_2025} or \cite{}).

For a partially hyperbolic flow $(\varphi^t)$ on a smooth compact manifold, the stable, unstable, and subcenter distributions are typically only continuous, so classical Frobenius-type integrability results do not apply. Nevertheless, the uniform contraction/expansion properties together with the invariance under $d\varphi^t$ ensure that $E^s$ and $E^u$ are uniquely integrable to flow-invariant foliations $\mathcal{F}^s$ and $\mathcal{F}^u$ respectively, whose leaves are as smooth as the flow.

The integrability of the subcenter bundle $E^{\hat{c}}$, however, is far less understood. Under the usual definition of partial hyperbolicity for flows (i.e., via the time-one map), questions about the center bundle's integrability reduce to the discrete case. There, it is known that even a smooth center bundle may not be integrable, and if a tangent foliation exists, it need not be unique or invariant \cite{wilkinson_stable_1998}.

Analogously to the discrete case, even assuming that $E^{\hat{c}}$ is integrable to a compact foliation does not imply that the subcenter foliation form a fibration, i.e. there might exists subcenter leaves with non-trivial holonomy group.

We call a compact foliation $\mathcal{F}$ \textit{uniformly compact} if every leaf has finite holonomy group. By \cite{epstein_foliations_1976}, this is equivalent to the quotient space $M/\mathcal{F}$ being Hausdorff; in that case, the quotient is a compact metric space under the Hausdorff metric and, if all holonomy groups are trivial, a topological manifold.

In the discrete setting, Bohnet (\cite{bohnet_partially_2011}) proved that if a partially hyperbolic diffeomorphism admits a uniformly compact invariant center foliation, then it is \textit{dynamically coherent}: there exist invariant foliations $\mathcal{F}^{cs}$ and $\mathcal{F}^{cu}$ tangent to $E^{cs} = E^s \oplus E^c$ and $E^{cu} = E^u \oplus E^c$, respectively. Dynamical coherence implies, in particular, that the center bundle itself is integrable to an invariant foliation $\mathcal{F}^c$, and that $\mathcal{F}^s$ and $\mathcal{F}^c$ subfoliate $\mathcal{F}^{cs}$, while $\mathcal{F}^u$ and $\mathcal{F}^c$ subfoliate $\mathcal{F}^{cu}$ (see Proposition 2.4 of \cite{wilkinson_dynamical_2008}). 
In fact, she proved that the center foliation is \textit{complete} in the sense that
$
\bigcup_{z \in \mathcal{F}^{c}(x)} \mathcal{F}^{*}(z) = \bigcup_{w \in \mathcal{F}^{*}(x)} \mathcal{F}^{c}(w)
$
for all $x \in M$ and $* \in \{s,u\}$. \\
Completeness implies dynamical coherence (see Proposition \ref{prop:dyncohcent}).

By analogy, we say a partially hyperbolic flow $(\varphi^t)$ is \textit{dynamically coherent} if there exist flow-invariant foliations $\mathcal{F}^{\hat{c}s}$ and $\mathcal{F}^{\hat{c}u}$ tangent to $E^{\hat{c}s} = E^s \oplus E^{\hat{c}}$ and $E^{\hat{c}u} = E^u \oplus E^{\hat{c}}$, respectively. As we will see (Proposition \ref{lem:weakfol}), this condition implies dynamical coherence of the time-one map $\varphi^1$, and ensures that the subcenter bundle integrates to a flow-invariant foliation $\mathcal{F}^{\hat{c}}$ such that $\mathcal{F}^s$ and $\mathcal{F}^{\hat{c}}$ subfoliate $\mathcal{F}^{\hat{c}s}$, while $\mathcal{F}^u$ and $\mathcal{F}^{\hat{c}}$ subfoliate $\mathcal{F}^{\hat{c}u}$ (Proposition \ref{prop:dyncohsubf}).
In that respect, we mention two results which arise quite immediately from the discrete case: we prove that a partially hyperbolic flow which is dynamical coherent admits global $su\Phi$ holonomy maps (Lemma \ref{lem:admitglobhol}), and if it is moreover subcenter-bunched, then $\mathcal{F}^s$ $C^1$-subfoliates $\mathcal{F}^{\hat{c}s}$ while $\mathcal{F}^u$ $C^1$-subfoliates $\mathcal{F}^{\hat{c}u}$ (Proposition \ref{prop:subunchC1}).
These two results are essential in our study of transversely holomorphic partially hyperbolic flows on seven-dimensional manifolds.

Finally, for a partially hyperbolic flow with an integrable subcenter bundle, we say the subcenter foliation is \textit{complete} if
$
\bigcup_{z \in \mathcal{F}^{\hat{c}}(x)} \mathcal{F}^{*}(z) = \bigcup_{w \in \mathcal{F}^{*}(x)} \mathcal{F}^{\hat{c}}(w)
$
for all $x \in M$ and $* \in \{s,u\}$. Again this implies in particular that the time-one map $\varphi^1$ has a complete center foliation and is thus dynamically coherent.

The main result of this article is:

\begin{customthm}{A}
\label{thm:subcentcomp}
Let $(\varphi^t)$ a partially hyperbolic flow on a smooth compact manifold $M$ whose subcenter distribution is integrable to a flow invariant compact foliation $\mathcal{F}^{\hat{c}}$ with trivial holonomy.\\
    Then $\mathcal{F}^{\hat{c}}$ is complete. 
    In particular, $(\varphi^t)$ is dynamically coherent.
\end{customthm}

In fact, the results seem to be true if we only assume $\mathcal{F}^{\hat{c}}$ uniformly compact, thanks to the works of Bonhet and Bonnati (see \cite{bohnet_partially_2011} and \cite{bohnet_partially_2014}).
We have chosen to prove it for the trivial holonomy assumption as the proofs are easier and can be generalized to the uniformly compact case.
Also, in our study of transversely holomorphic partially hyperbolic flows on smooth compact seven-dimensional manifolds, we only need $\mathcal{F}^{\hat{c}}$ to be compact with trivial holonomy so that the leaf space $M/\mathcal{F}^{\hat{c}}$ is a topological manifold (and thus a smooth manifold, see Theorem ?? of \cite{}).

\section{Definitions}
We only consider smooth ($C^\infty$) connected manifolds without boundary.
\subsection{Foliations}
We start by recalling some notions coming from the theory of foliations (see \cite{camacho_geometric_1985}, \cite{moerdijk_introduction_2003}, \cite{lee_manifolds_2009} and \cite{candel_foliations_2000}). 
\begin{defprop}
\label{defprop:fol}
    Let $M$ a smooth manifold of dimension $n\in \mathbb N$. Let $r \geq 1$ (or $r\in \{\infty\}$) and $k\leq n$.\\
    A \emph{$C^r$ foliation} $\mathcal{F}$ of dimension $k$ on $M$ (or codimension $n-k$) is defined by one of the following equivalent assertions:
    \begin{enumerate}
        \item a $C^r$-atlas $(U_i, \psi_i)_i$ on $M$ which is maximal with respect to the following properties:
            \begin{enumerate}
                \item For all $i$, $\psi_i(U_i)=U_i^1 \times U_i^2$, where $U_i^1$ and $U_i^2$ are connected open subsets of $\mathbb R^k$ and $\mathbb R^{n-k}$ respectively ;
                \item For all $i,j$, there exist $C^r$ maps $f_{ij}$ and $h_{ij}$ such that
                \[\forall(x,y) \in \psi_i(U_i \cap U_j) \subset \mathbb R^k \times \mathbb R^{n-k}, \; \psi_i \circ \psi_j^{-1}(x,y)=(f_{ij}(x,y), h_{ij}(y)).\]
            \end{enumerate}
        \item a maximal atlas $(U_i, s_i)_i$, where each $U_i$ is an open subset of $M$ and $s_i:U_i \to \mathbb R^{n-k}$ is a $C^r$ submersion, satisfying:
            \begin{enumerate}
                \item $\bigcup_i U_i = M$ ;
                \item For all $i,j$, there exists a $C^r$ diffeomorphism
                \[\gamma_{ij}: s_i(U_i\cap U_j) \to s_j(U_i\cap U_j) \text{  such that  } \restr{s_i}{U_i\cap U_j} = \restr{\gamma_{ij}\circ s_j}{U_i\cap U_j}. \]
            \end{enumerate}
        \item a partition of $M$ into a family of disjoint connected $C^r$ immersed submanifolds $(L_\alpha)_\alpha$ of dimension $k$  such that for every $x \in M$, there is a $C^r$ chart $(U, \psi)$ at $x$, of the form $\psi: U \to U^1 \times U^2$ where $U^1$ and $U^2$ are connected open subsets of $\mathbb R^k$ and $\mathbb R^{n-k}$ respectively, satisfying: for each $L_\alpha$, for each connected component $(U \cap L_\alpha)_\beta$ of $U\cap L_\alpha$, there exists $c_{\alpha,\beta} \in \mathbb R^{n-k}$ such that 
        \[ \psi ((U\cap L_\alpha)_\beta) = U^1 \times \{c_{\alpha, \beta}\}.\]
     \end{enumerate}
    The maps $h_{ij}$ and $\gamma_{ij}$ are called \textit{transition maps}.\\
    We will call (for each equivalent assertion $(1)$, $(2)$ and $(3)$):
    \begin{itemize}
        \item a \emph{plaque}:
            \begin{enumerate}
                \item a set of the form $\psi_i^{-1}(U_i \times \{c\})$, for $c \in U_i^2$ ;
                \item a connected component of a fiber of $s_i$ ;
                \item a connected component $(U\cap L_\alpha)_\beta$ of $U\cap L_\alpha$.
            \end{enumerate}
        \item a \emph{leaf} of the foliation an equivalence class for the following equivalence relation: two points $x$ and $y$ are equivalent if and only if there exist a sequence of foliation charts $U_1, \ldots, U_k$ and a sequence of points $x=p_0, p_1, \ldots, p_{k-1},p_k=y$ such that, for $i\in\llbracket 1,k\rrbracket$, $p_{i-1}$ and $p_{i}$ lie on the same plaque of $U_i$.
    \end{itemize}
    \end{defprop}
    We will represent abusively a foliation by the set of its leaves $\mathcal{F}$ and will note $\mathcal{F}_x$ or $\mathcal{F}(x)$ the leaf of the foliation containing $x\in M$. \\
    In case of the orbit foliation of a non-vanishing smooth vector field, we will write $\Phi$. 
    \begin{remark}
    \label{rem:2.1}
    \begin{itemize}
    \leavevmode
        \item A $C^r$ foliation $\mathcal{F}$ gives rise to a $C^{r-1}$ subbundle of $TM$, noted $T\mathcal{F}$, called the \textit{tangent bundle} to the foliation $\mathcal{F}$, whose fibers are the tangent spaces to the leaves. We also define the \textit{normal bundle} $\nu:= TM \diagup T\mathcal{F}$ of the foliation. It is a vector bundle whose transition functions are given by the differential of the transition maps $h_{ij}$.
        \item We can define in the same way a \emph{holomorphic} foliation on a complex manifold by replacing "$C^r$" in the above definition with "holomorphic" and $\mathbb R$ with $\mathbb C$.
    \end{itemize}
    \end{remark}

From the first and third formulations we can define a more general notion of foliation adapted to our purposes: 
\begin{definition}
\label{def:hol}
    Let $M$ a smooth manifold of dimension $n\in \mathbb N$. Let $r \geq 1$ (or $r=\infty$) and $k\leq n$.\\
    A \emph{(topological) foliation} with $C^r$ leaves, $\mathcal{F}$, of dimension $k$ (or codimension $n-k$) on $M$ is defined by one of the following equivalent assertions:
    \begin{enumerate}
        \item a $C^0$-atlas $(U_i, \psi_i)_i$ on $M$ which is maximal with respect to the following properties:
            \begin{enumerate}
                \item For all $i$, $\psi_i(U_i)=U_i^1 \times U_i^2$, where $U_i^1$ and $U_i^2$ are connected open subsets of $\mathbb R^k$ and $\mathbb R^{n-k}$ respectively ;
                \item For all $i,j$, there exist $C^0$ maps $f_{ij}$ and $h_{ij}$ such that
                \[\forall(x,y) \in \psi_j(U_i \cap U_j) \subset \mathbb R^k \times \mathbb R^{n-k}, \; \psi_i \circ \psi_j^{-1}(x,y)=(f_{ij}(x,y), h_{ij}(y))\]
                and for every such $y$, the map $x\mapsto f_{ij}(x,y)$ is $C^r$.
            \end{enumerate}
        \item a partition of $M$ into a family of disjoint connected $C^r$ immersed submanifolds $(L_\alpha)_\alpha$ of dimension $k$  such that for every $x \in M$, there is a $C^0$ chart $(U, \psi)$ at $x$, of the form $\psi: U \to U^1 \times U^2$ where $U^1$ and $U^2$ are connected open subsets of $\mathbb R^k$ and $\mathbb R^{n-k}$ respectively, satisfying: for each $L_\alpha$, for each connected component $(U \cap L_\alpha)_\beta$ of $U\cap L_\alpha$, there exists $c_{\alpha,\beta} \in \mathbb R^{n-k}$ such that 
        \[ \psi ((U\cap L_\alpha)_\beta) = U^1 \times \{c_{\alpha, \beta}\}\]
        and the map $\operatorname{pr}_1\circ \psi|_{(U \cap L_\alpha)_\beta}:(U \cap L_\alpha)_\beta\to U^1$ is a $C^r$ chart for $L$.
     \end{enumerate}
     We call such an atlas a \emph{foliated atlas} adapted to $\mathcal{F}$. 
\end{definition}

\begin{remark}
\label{rem:fol}
    \leavevmode
        \begin{itemize}
        \item For every $y$, the map $x\mapsto f_{ij}(x,y)$ in the previous definition is automatically a $C^r$ diffeomorphism.  
        \item We can still define the tangent bundle $T\mathcal{F}$ to the foliation $\mathcal{F}$ if it is a foliation with $C^r$ leaves, $r\geq 1$.
        It is a vector bundle over $M$ whose transition maps are given by the differential along the first variable of the maps $f_{ij}$, ie by $d_1f_{ij}$.
        \item In case $T\mathcal{F}$ is a continuous subbundle of $TM$, the foliation $\mathcal{F}$ is said to be \emph{integral}. This is the same as saying that the partial derivatives of $f_{ij}$ with respect to $x$ vary continuously with $(x,y)$, i.e. $\mathcal{F}$ has \emph{uniformly $C^1$ leaves} (see \cite{pugh_holder_1997} and \cite{wilkinson_dynamical_2008} for a discussion about integrability of continuous subbundles of $TM$).
        In order to facilitate the statement of some results, we will call integral foliations \emph{$C^0$ foliations}.\\
        Note that every $C^r$ foliation, $r\geq 1$, is integral, and by definition any integral foliation has $C^1$ leaves.
        \item We can define in the same way (if $k=2k'$ is even) a \emph{foliation with holomorphic leaves} by replacing "$C^r$" in the above definition with "holomorphic".
        \end{itemize}
\end{remark}

As mentioned in \cite{bohnet_partially_2014}, an integral foliation $\mathcal{F}^*$ on a smooth compact manifold $M$ has the following properties. Denote by $d^*$ the distance in a leaf $L$ of $\mathcal{F}^*$ coming from the induced metric of $M$ on $L$. Fix a finite foliated atlas $(U_i,\psi_i)$ for $\mathcal{F}$. Then:

\begin{itemize}
    \item For every $\eta > 0$, there is $\mu$ so that, if $x,y$ are points in the same plaque such that $d(x,y) < \mu$, then and $d^{*}(x,y) < (1+\eta)d(x,y)$. As a consequence, there exists $\Delta^{*} > 0$ such that for every $i$ and every $x,y \in U_{i}$, if $d^{*}(x,y) < \Delta^{*}$, then $x$ and $y$ are in the same plaque.
    \item If $\mathcal{F}^{*}$ is a compact foliation (that is, every leaf is compact), then each of its leaf intersects any foliation chart in a finite number of plaques (see \cite{camacho_geometric_1985}). Therefore, two points in the same leaf which are close enough in the ambient manifold are in the same plaque. In that case, the above can be reformulated as:\\
    For every $\eta > 0$, there is $\mu > 0$ so that, if $x,y$ are points in the same leaf such that $d(x,y) < \mu$, then $d^{*}(x,y) < (1+\eta)d(x,y)$.
\end{itemize}

From now on, we will always consider $C^0$ foliations on a smooth manifold $M$.

\subsection{Holonomy}
    
We now define the fundamental concept of \textit{holonomy} for an integral foliation (see \cite{moerdijk_introduction_2003}, \cite{camacho_geometric_1985}).
In the following, a $C^0$ diffeomorphism means a homeomorphism. 

\begin{definition}
    A \emph{transversal} to the (integral) foliation $\mathcal{F}$ is a smooth submanifold $T$ of $M$ such that for every $x\in T$, $T_xT\oplus T_x\mathcal{F}=T_xM$.
\end{definition}
\begin{defprop}
\label{defprop:hol}
    Let $ \mathcal{F}$ a $C^r$ foliation on a smooth manifold $M$ ($r \in \mathbb N \cup \{\infty\}$). Let $x,y$ belonging to the same leaf $L$ of $\mathcal{F}$ and $\alpha$ a path from $x$ to $y$ in $L$. Let also $T,S$ two small transversal to $\mathcal{F}$ at $x$ and $y$ respectively.
    \begin{itemize}
        \item If there exists a foliation chart $U$ containing $\alpha([0,1])$, we can define a germ of $C^r$ diffeomorphism $h(\alpha)^{S,T}$ from a small open subset $A$ of $T$ to an open subset of $S$ such that: \begin{enumerate}
            \item $h(\alpha)^{S,T}(x)=y$ ;
            \item For any $x'\in A$, $h(\alpha)^{S,T}(x')$ lies on the same plaque in $U$ as $x'$.
        \end{enumerate}
        Moreover, the germ $h(\alpha)^{S,T}$ does not depend on $U$ nor the path in $L\cap U$ connecting $x$ and $y$.
        \item In the general case, we can choose a sequence of foliation charts $U_1, \ldots, U_k$ such that for all $i$, $\alpha([\frac{i-1}{k}, \frac{i}{k}]) \subset U_i$, and a sequence $(T_i)_i$ of transversal sections of $\mathcal{F}$ at $\alpha(\frac{i}{k})$, with $T_0=T$ and $T_k=S$. Let $\alpha_i$ a path in $L\cap U_i$ from $\alpha(\frac{i-1}{k})$ to $\alpha(\frac{i}{k})$. \\
        We define 
        \[h(\alpha)^{S,T}:=h(\alpha_k)^{T_k, T_{k-1}} \circ \cdots \circ h(\alpha_1)^{T_1, T_0}.
        \]
        Moreover, $h(\alpha)^{S,T}$ depends only on $T, S$ and $\alpha$, and satisfies the same properties $(1), (2)$.
    \end{itemize}
    The germ of $C^r$ diffeomorphism $h(\alpha)^{S,T}$ is called the \emph{$\mathcal{F}$-holonomy} of the path $\alpha$ with respect to the transversals $T$ and $S$.
\end{defprop} 
    We list some basic properties of holonomy maps:
    \begin{enumerate}[label=(\roman*)]
        \item If $\alpha$ is a path in $L$ from $x$ to $y$ and $\beta$ a path in $L$ from $y$ to $z$, and if $T$, $S$ and $R$ are transversal sections of $\mathcal{F}$ at $x$, $y$ and $z$ respectively, then
        \[h(\beta \alpha)^{R,T}=h(\beta)^{R,S} \circ h(\alpha)^{S,T},\]
        where $\beta\alpha$ is the concatenation of the paths $\alpha$ and $\beta$ ;
        \item If $\alpha$ and $\beta$ are homotopic paths in $L$ relatively to endpoints from $x$ to $y$, and if $T$ and $S$ are transversals sections at $x$ and $y$ respectively, then $h(\beta)^{S,T}= h(\alpha)^{S,T}$.
        \item If $\alpha$ is a path in $L$ from $x$ to $y$, and if $T,T'$ and $S,S'$ are pairs of transverse sections at $x$ and $y$, respectively, then \[h(\alpha)^{S',T'}=h(\overline{y})^{S',S} \circ h(\alpha)^{S,T} \circ h(\overline{x})^{T,T'},\] where $\overline{x}$ is the constant path with image $x$.
    \end{enumerate}
    \begin{remark}
    \label{rem:hol}
    \begin{itemize}
    \leavevmode
        \item The construction of the holonomy diffeomorphism for two points in the same plaque of the same foliation chart $(U, \psi)$ is as follows. Keep the notations of Definition/Proposition \ref{defprop:hol}. Denote by $n$ the dimension of $M$ and $q$ the codimension of $\mathcal{F}$. 
        $\psi(T) , \psi(S)\subset \mathbb R^{n}$ are vertical graphs of $C^r$ maps from $\mathbb R^q$ to $\mathbb R^{n-q}$.
        By projecting these graphs on $\mathbb R^q$, we obtain two $C^r$ diffeomorphisms $\Psi: T \to \mathbb R^q$ and $\Psi': S \to \mathbb R^q$. Since the leaves of the foliation are the horizontal lines, it comes 
        \[h(\alpha)^{S,T}=\Psi'^{-1} \circ \Psi.\]
        This means that under these coordinates, the holonomy between two points in the same plaque of the same foliation chart is equal to the identity.
        \item The holonomy of a path in a leaf with respect to the natural transversals given by the foliation charts can be expressed in local coordinates as a composition of transition functions $h_{ij}$ (see Definition/Proposition \ref{defprop:fol}). More precisely, let $\alpha$ a path in $L$ from $x$ to $y$ and a sequence of foliation charts $U_1, \ldots, U_k$ such that for all $i$, $\alpha([\frac{i-1}{k}, \frac{i}{k}]) \subset U_i$ and $U_1$ and $U_k$ are centered at $x$ and $y$ respectively.
        Let $T=\psi_1^{-1}(\{0\} \times U_1^2)$ and $S=\psi_k^{-1}(\{0\} \times U_k^2)$.
        Then there exists a small open subset $V$ of $T$ such that 
        \[h(\alpha)^{S,T}= \restr{\psi_k^{-1} \circ h_{k,k-1} \circ \cdots \circ h_{2,1} \circ \psi_1}{V}.
        \]
    \end{itemize}
    \end{remark}
    
\begin{definition}
    Given a $C^r$ foliation $\mathcal{F}$ on a manifold $M$, the set of all holonomy maps of any path with respect to any transversals to $\mathcal{F}$ defines a pseudo-group called the \emph{holonomy pseudo-group} of the foliation $\mathcal{F}$.\\
    For every leaf $L$ of $\mathcal{F}$, $x\in L$ and small transversal section $T$ at $x$, there is a group homomorphism 
    \[h^T:=h^{T,T}: \pi_1(L,x) \to \text{Diff}^r_x(T), \]
    where $\text{Diff}^r_x(T)$ is the group of $C^r$ diffeomorphisms of $T$ which fix $x$, 
    called the \emph{holonomy homomorphism} of $L$, determined up to conjugation.\\
    The set $h^T(\pi_1(L,x))$ is a group called the \textit{holonomy group} of the leaf $L$, determined also up to conjugation.\\
    We define the space $\widetilde{L}\diagup \ker(h^T)$, where $\widetilde{L}$ is the universal covering space of $L$. It does not depend on $x \in L$ nor $T$ and is called the \emph{holonomy covering} of $L$.
    \end{definition}
    
    We say that a $C^r$ foliation $\mathcal{F}$ has \textit{trivial} (respectively \emph{finite}) \textit{holonomy} if the holonomy group of every leaf of $\mathcal{F}$ is trivial (respectively finite).
    
\subsection{Compact foliations}
We recall some results about compact foliations, that is foliations whose leaves are compact.\\
The following result is proven in \cite{carrasco_compact_2015}. 
\begin{theorem}
\label{thm:carra}
    Let $\mathcal{F}$ a compact $C^0$ foliation on a smooth manifold $M$.\\
    Then the following assertions are equivalent:
    \begin{enumerate}[label=(\roman*)]
        \item The quotient map $\pi: M\to M/\mathcal{F}$ is closed.
        \item $M/\mathcal{F}$ is Hausdorff.
        \item Every leaf of $\mathcal{F}$ has arbitrarily small saturated neighborhoods.
        \item The saturation of every compact subset of $M$ is compact.
    \end{enumerate}
    If $M$ is compact, then the above properties are also equivalent to:
    \begin{enumerate}[label=(\roman*),start=5]
        \item The holonomy group of every leaf of $\mathcal{F}$ is finite.
        \item The volume of the leaves is uniformly bounded from above.
    \end{enumerate}
    Moreover if $\mathcal{F}$ has trivial holonomy, then the resulting leaf space $M/\mathcal{F}$ is a topological manifold.
\end{theorem}
We say that a compact foliation is \emph{uniformly compact} if the holonomy group of each of its leaves is finite.\\
In fact, the equivalence with the two last assertions is due to the following theorem (see \cite{candel_foliations_2000} for more details):
\begin{theorem}[Generalized Reeb Stability]
\label{thm:reeb}
    Let $L$ a compact leaf of a $C^r$ foliation on a smooth manifold $M$. Assume its holonomy group $\operatorname{Hol}(L)$ is finite.\\
    Then there exists a normal neighborhood $p:V\to L$ of $L$ in $M$ which is a $C^r$ fiber bundle with structure group $\operatorname{Hol}(L)$.\\
    Furthermore, each leaf $L'|_V$ is a covering space $p|_{L'}:L' \to L$ with $k\leq |\operatorname{Hol}(L)|$ sheets and the leaf $L'$ has a finite holonomy group of order $\frac{\operatorname{Hol}(L)}{k}$.
\end{theorem}
A corollary of this result, which motivates our assumption of trivial holonomy for the subcenter foliation, is that the leaves with trivial holonomy of a compact foliation are \emph{generic} (see \cite{candel_foliations_2000}):
\begin{corollary}
    Let $\mathcal{F}$ a compact $C^0$ foliation.\\
    Then the set of leaves with trivial holonomy is open and dense.
\end{corollary}

\subsection{Partially hyperbolic flows}
See \cite{hirsch_invariant_1977}, \cite{brin_partially_1974} and \cite{barreira_nonuniform_2007} for more details.
\begin{definition}
\label{def:parthypflo}
    Let $M$ a smooth compact manifold and $(\varphi^t)_{t \in \mathbb R}$ a smooth flow on $M$  generated by a nowhere vanishing smooth vector field $X$.\\
   The flow $(\varphi^t)$ is said to be \textit{(uniformly) partially hyperbolic} on $M$ if there exist $(d\varphi^t)$-invariant continuous subbundles of $TM$ - $E^{s}$, $E^{\hat{c}}$ and $E^{u}$ - a Riemannian metric $\|.\|$ on $M$, real numbers $\alpha <\alpha'< 1 < \beta'<\beta$ and $C>0$ such that:
    \begin{enumerate}
        \item $TM = E^{s} \oplus E^{\hat{c}} \oplus  \mathbb R X \oplus E^{u}$ ;
        \item For every $t\geq 0$:
        \label{eq:ano}
            \begin{alignat*}{2}
                &\forall v \in E^{s}, \quad &&\|d\varphi^t(v)\| \leq C\alpha^t\|v\|,\\
                &\forall v \in E^{u}, \quad &C^{-1}\beta^{t}\|v\| \leq{} &\|d\varphi^{t}(v)\|,\\
                &\forall v \in E^{\hat{c}}, \quad &C^{-1}\alpha'^{t}\|v\| \leq{} &\|d\varphi^{t}(v)\| \leq C\beta'^{t}\|v\| .
            \end{alignat*}
    \end{enumerate}
    The vector bundles $E^{s}$, $E^{\hat{c}}$ and $E^{u}$ are called respectively the \emph{strong stable, subcenter, \emph{and} strong unstable distributions} of $(\varphi^t)$, while $\mathbb R X \oplus E^{s}$, $E^c:=\mathbb R X \oplus E^{\hat{c}}$ and $\mathbb R X \oplus E^{u}$ are called respectively the \emph{weak stable, center \emph{and} weak unstable distributions} of $(\varphi^t)$.
\end{definition}
Remark that for every $x \in M$ and $t\geq 0$, $d_x\varphi^t(X(x))=X(\varphi^t(x))$, so if we note $C_1,C_2>0$ such that for every $x \in M$, $C_1 \leq \|X(x)\|\leq C_2$ (which exist since $M$ is compact and $X$ is continuous nowhere vanishing), it comes that:
$$\forall v \in \mathbb RX(x), \quad \frac{C_1}{C_2}\|v\| \leq{} \|d_x\varphi^{t}(v)\| \leq\frac{C_2}{C_1}\|v\| .$$
In particular, the time-$t$ map of a partially hyperbolic flow is a partially hyperbolic diffeomorphism whose stable, center and unstable distributions are $E^s, E^c$ and $E^u$ respectively.\\
This definition recovers that of an \textit{Anosov flow}, that is when the subcenter bundle $E^{\hat{c}}$ is trivial.\\
Via a change of Riemannian metric we can assume that $C = 1$ (see \cite{barreira_nonuniform_2007}).

Our definition of partially hyperbolic flow is identical to that of \cite{wang_quasi-shadowing_2023}, or \cite{carneiro_partially_2014}, and has the advantage of considering explicitly a subbundle of $TM$ which is transverse to the vector field inducing the flow. 
Therefore, one of the main techniques used in order to study a flow - which is assuming a transverse structure invariant by flow holonomies - can be very efficient: in \cite{}, we study partially hyperbolic transversely holomorphic flows on $7$-dimensional manifolds whose subcenter distribution is integrable to a flow-invariant foliation with $C^1$ leaves and trivial holonomy and we prove that their study is analogous of that of transversely holomorphic Anosov flows on $5$-dimensional manifolds, which are classified in \cite{abouanass_global_2025}.
Also, it is sometimes difficult to extrapolate some techniques used in the discrete case: in \cite{bohnet_partially_2011} and \cite{bohnet_partially_2014}, the authors study partially hyperbolic diffeomorphisms whose center distribution is integrable (see the next subsection) to a compact foliation with finite holonomy.
In our case, such a phenomenon is impossible as the center foliation would be subfoliated by the orbit foliation which possesses non-compact leaves (the non periodic orbits).

\subsection{Dynamical coherence}
As was discussed in \cite{wilkinson_dynamical_2008}, one can define different notions of integrability of a continuous subbundle of $TM$.
We recall some of them.
Let $E$ a continuous subbundle of dimension $k$ of the tangent bundle of a smooth manifold $M$.
$E$ is said to be \textit{integrable} if there exists a foliation $\mathcal{F}$ of $M$ by $C^1$ immersed submanifolds of dimension $k$ which are everywhere tangent to $E$.
$E$ is said to be \textit{uniquely integrable} if it is integrable to a foliation $\mathcal{F}$ and every $C^1$ path of $M$ everywhere tangent to $E$ lies in a single leaf of $\mathcal{F}$.
The notion of unique integrability is different from the fact that there exists a unique foliation tangent to $E$.
If $E$ is uniquely integrable, then there exists a unique foliation whose leaves are everywhere tangent to $E$.

The stable and unstable bundles of a partially hyperbolic flow are uniquely integrable to invariant foliations $\mathcal{F}^s$ and $\mathcal{F}^u$ respectively, whose leaves are as smooth as the flow (see \cite{hirsch_invariant_1977} or \cite{barreira_nonuniform_2007}).

\begin{definition}
    We say that a partially hyperbolic flow $(\varphi^t)$ is \textit{dynamically coherent} if there exist flow-invariant foliations $\mathcal{F}^{\hat{c}s}$ and $\mathcal{F}^{\hat{c}u}$ whose leaves are everywhere tangent to $E^{\hat{c}s}=E^{s}\oplus E^{\hat{c}}$ and $E^{\hat{c}u}=E^{u}\oplus E^{\hat{c}}$ respectively.
\end{definition}

This implies in particular, by Lemma \ref{lem:weakfol}, that the time-one map $\varphi^1$ is dynamically coherent.
Also, we can proceed as in \cite{wilkinson_dynamical_2008} to prove:

\begin{proposition}
\label{prop:dyncohsubf}
    Let $(\varphi^t)$ a partially hyperbolic flow on a smooth compact manifold which is dynamically coherent.\\
    Then the subcenter distribution is integrable to a flow-invariant foliation $\mathcal{F}^{\hat{c}}$. 
    Moreover, $\mathcal{F}^s$ and $\mathcal{F}^{\hat{c}}$ subfoliate $\mathcal{F}^{\hat{c}s}$ while $\mathcal{F}^u$ and $\mathcal{F}^{\hat{c}}$ subfoliate $\mathcal{F}^{\hat{c}u}$.
\end{proposition}
\begin{proof}
    The proof is identical to that of Proposition 2.4 of \cite{wilkinson_dynamical_2008}.
    We first prove that any disk tangent to $E^{\hat{c}u}$ is subfoliated by $\mathcal{F}^{u}$-plaques.
    Let $D$ such a disk and $X$ the disjoint union of all the $\varphi^1$-iterates of $D$, called the leaves of $X$.
    On the manifold $X$ there is a partially hyperbolic diffeomorphism $F$ whose stable bundle is trivial and whose center bundle, when restricted to a leaf $L$ of $X$, is the restriction to $L$ of the subcenter bundle of $(\varphi^t)$.
    We conclude as in Lemma 2.5 and finish the proof by intersecting the leaves of $\mathcal{F}^{\hat{c}u}$ and $\mathcal{F}^{\hat{c}s}$ to obtain $\mathcal{F}^{\hat{c}}$, which is flow-invariant since $\mathcal{F}^{\hat{c}s}$ and $\mathcal{F}^{\hat{c}u}$ are both flow-invariant.
\end{proof}
\begin{definition}
For a partially hyperbolic flow $(\varphi^t)$ whose subcenter distribution is integrable to a foliation $\mathcal{F}^{\hat{c}}$, we say that its subcenter foliation is \textit{complete} if for every $x \in M$ and $* \in \{s,u\}$, $\bigcup_{z \in \mathcal{F}^{\hat{c}}(x)} \mathcal{F}^{*}(z)=\bigcup_{w \in \mathcal{F}^{*}(x)} \mathcal{F}^{\hat{c}}(w)$.
\end{definition}
Again, as we will see later, if the subcenter foliation of $(\varphi^t)$ is complete, then the time-one map $\varphi^1$ has a complete center foliation and is thus dynamically coherent.

\subsection{Global $su\Phi$-holonomy maps}

If $(\varphi^t)$ is a dynamically coherent partially hyperbolic flow, then each leaf of $\mathcal{F}^{\hat{c}s}$ is simultaneously subfoliated by $\mathcal{F}^{\hat{c}}$ and $\mathcal{F}^{s}$.
In particular, for any $x\in M$, we can define the holonomy of the continuous foliation $\mathcal{F}^s|_{\mathcal{F}^{\hat{c}s}(x)}$ of $\mathcal{F}^{\hat{c}s}(x)$, which we call the local stable holonomy. 
For $y \in \mathcal{F}^s(x)$, since stable leaves are contractible, the local stable holonomy of a path $\gamma$ in $\mathcal{F}^s$ between $x$ and $y$ with respect to $\mathcal{F}_{\text{loc}}^{\hat{c}}(x)$ and $\mathcal{F}_{\text{loc}}^{\hat{c}}(y)$ does not depend on $\gamma$.

\begin{definition}
    A dynamically coherent partially hyperbolic flow $(\varphi^t)$ admits \emph{global stable holonomy maps} if for any $x\in M$, $y\in \mathcal{F}^{s}(x)$, there exists a globally defined homeomorphism $h_{xy}^{s}: \mathcal{F}^{ \hat{c}}(x)\to \mathcal{F}^{\hat{c}}(y)$ such that for every $z\in \mathcal{F}^{\hat{c}}(x)$, $h_{xy}^{s}(z)\in \mathcal{F}^{s}(z)\cap \mathcal{F}^{\hat{c}}(y)$.\\
    Similarly, we can define global unstable holonomy maps $h^{u}$ and global flow holonomy maps $h^\Phi$.\\
    We say that $(\varphi^t)$ admits \emph{global $su\Phi$-holonomy maps} if it admits global stable, global unstable and global flow holonomy maps.
\end{definition}
Since global holonomy maps coincide locally with local holonomy maps, we use $h_{xy}^{s}$ to denote both local holonomy maps and global holonomy maps. 
If $\mathcal{F}^{\hat{c}}$ is invariant by the flow, then the latter admits global flow holonomy maps as we can define, for $x \in M$ and $y =\varphi^{t_0}(x) \in \Phi(x)$, the homeomorphism $z \in \mathcal{F}^{\hat{c}}(x) \mapsto \varphi^{t_0}(z) \in \mathcal{F}^{\hat{c}}(y)$ which coincides locally with local flow holonomy maps.
In that respect, by analogy with the discrete case, we can prove:

\begin{lemma}
\label{lem:subleavhomeo}
    If $(\varphi^t)$ admits global $su\Phi$-holonomy maps, then all subcenter leaves are homeomorphic. 
\end{lemma}
\begin{proof}
    We show that for every $x \in M$, there exists a neighborhood $U_x$ of $x$ in $M$ such that for every $y \in U_x$, $\mathcal{F}^{\hat{c}}(y)$ is homeomorphic to $\mathcal{F}^{\hat{c}}(x)$. 
    The result will follow since we can always consider a finite cover of $M$ by such neighborhoods (since $M$ is compact).
    Let a foliation chart $V$ for $\mathcal{F}^s|_{\mathcal{F}^{\hat{c}s}(x)}$ centered at $x \in \mathcal{F}^{\hat{c}s}(x)$ (which exists since $\mathcal{F}^s$ subfoliates $\mathcal{F}^{\hat{c}s}$ by Proposition \ref{prop:dyncohsubf}).
    Let $\mathcal{F}_{\text{loc}}^{\hat{c}}(x)$ a small open neighborhood of $x$ in $\mathcal{F}^{\hat{c}}(x)$ which lies in $V$.
    For every $w \in \mathcal{F}_{\text{loc}}^{\hat{c}}(x)$, denote by $P_w$ the associated $\mathcal{F}^s$ plaque.
    The set $\mathcal{F}_{\text{loc}}^{\hat{c}s}(x):=\bigcup_{w \in \mathcal{F}_{\text{loc}}^{\hat{c}}(x)}P_w$ is a small open neighborhood of $x$ in $\mathcal{F}^{\hat{c}s}(x)$.
    Therefore, $U_x:=\bigcup_{z \in \mathcal{F}_{\text{loc}}^{\hat{c}s}(x)}\mathcal{F}_{\text{loc}}^{wu}(z)$ is a small open neighborhood of $x$ in $M$.
    If $y \in U_x$, there exist  $w \in \mathcal{F}_{\text{loc}}^{\hat{c}}(x)$, $z \in P_w$ and a small number $t$ such that $\varphi^t(y) \in \mathcal{F}^{u}_{\text{loc}}(z)$.
    By assumption, the global flow holonomy map gives a homeomorphism between $\mathcal{F}^{\hat{c}}(y)$ and $\mathcal{F}^{\hat{c}}(\varphi^t(y))$ ; the global unstable holonomy map gives a homeomorphism between $\mathcal{F}^{\hat{c}}(\varphi^t(y))$ and $\mathcal{F}^{\hat{c}}(z)$ and the global stable holonomy map gives a homeomorphism between $\mathcal{F}^{\hat{c}}(z)$ and $\mathcal{F}^{\hat{c}}(w)=\mathcal{F}^{\hat{c}}(x)$, which gives the result.
\end{proof}

\begin{lemma}
\label{lem:admitglobhol}
    Let $(\varphi^t)$ a partially hyperbolic flow on a smooth compact manifold $M$, with a flow invariant subcenter foliation $\mathcal{F}^{\hat{c}}$ with $C^1$-leaves and trivial holonomy.\\
    Then $(\varphi^t)$ admits global $su\Phi$-holonomies.
\end{lemma}
\begin{proof}
    As we mentioned right before, $(\varphi^t)$ already admits global flow holonomy maps.
    The proof that it admits both global stable and global unstable holonomy maps is analogous to the discrete case and can be found in \cite{avila_absolute_2022}, Lemma 3.5.
\end{proof}

\subsection{Subcenter bunching}
\begin{definition}
    Let $(\varphi^t)$ a smooth partially hyperbolic flow on a smooth compact manifold $M$ such that there exists a flow-invariant subcenter foliation $\mathcal{F}^{\hat{c}}$ with $C^1$-leaves.\\
    We say that $(\phi^t)$ is \textit{subcenter bunched} if there exists $t_0>0$ such that for every $p \in M$,
    \begin{align*}
        \left\|\restr{d_p\varphi^{t_0}}{E^{s}(p)} \right\| \cdot&\left\|\restr{d_p\varphi^{t_0}}{E^{\hat{c}}(p)} \right\| < m\left( \restr{d_p\varphi^{t_0}}{E^{\hat{c}}(p)} \right)\\
        \left\|\restr{d_p\varphi^{t_0}}{E^{\hat{c}}(p)} \right\|<&\;m\left( \restr{d_p\varphi^{t_0}}{E^{\hat{c}}(p)} \right)\cdot m\left( \restr{d_p\varphi^{t_0}}{E^{u}(p)} \right).
    \end{align*}
\end{definition}
Again, by analogy with the discrete case, we prove:

\begin{proposition}
\label{prop:subunchC1}
     Let $(\varphi^t)$ a smooth partially hyperbolic dynamically coherent flow which is subcenter bunched.\\
     Then for every $p \in M$, the restriction $\restr{\mathcal{F}^u}{\mathcal{F}^{\hat{c}u}(p)}$ to $\mathcal{F}^{\hat{c}u}(p)$ of the strong unstable foliation (respectively the restriction $\restr{\mathcal{F}^s}{\mathcal{F}^{\hat{c}s}(p)}$ to $\mathcal{F}^{\hat{c}s}(p)$ of the strong stable foliation) is $C^1$.
\end{proposition}
\begin{proof}
    The proof is identical to that of \cite{pugh_holder_1997} Theorem B by considering the partially hyperbolic diffeomorphism $\varphi^{t_0}: M \to M$ where $t_0$ is given by the subcenter-bunching assumption.
    The only difference comes at the end. In page 539, the base space is $\mathcal{F}^{\hat{c}u}(p)$.
    Therefore, for $z \in \mathcal{F}^{\hat{c}u}(p)$, the base is contracted by $$m(\left(d_z\varphi^{t_0}|_{E^{cs}(z)}\right)|_{E^{\hat{c}u}(z)} )=m(\left(d_z\varphi^{t_0}\right)|_{E^{\hat{c}}(z)}),$$ $\varphi^{t_0}$ overflows the base, and the fiber is contracted by $$\left\|(d_z\varphi^{t_0}|_{E^{cs}(z)})|_{E^{\hat{c}u}(z)}) \right\| \cdot \left (m(\left(d_z\varphi^{t_0}|_{E^{u}(z)}\right)|_{E^{\hat{c}u}(z)} ) \right)^{-1}=\left\|\left(d_z\varphi^{t_0}\right)|_{E^{\hat{c}}(z)} \right\| \cdot \left ( m\left(d_z\varphi^{t_0}|_{E^{u}(z)}\right) \right)^{-1}.$$
    Therefore, one can apply the invariant section theorem as in \cite{pugh_holder_1997} and conclude.
\end{proof}

\section{Sketch of proof of the results}
In section \ref{sec:4}, we state and prove some fundamental general results on foliations transverse to flows which will allow us to extrapolate some of the work done in the discrete case in ours. The main result is Lemma \ref{lem:weakfol} which gives, for an invariant foliation $\mathcal{F}$ with $C^1$-leaves transverse to the direction of the flow, the existence of a foliation tangent to $T\mathcal{F}\oplus \mathbb R X$.

In section \ref{sec:5}, we prove that the center foliation is complete if the subcenter foliation is compact with trivial holonomy.
We use a result from Bonhet and Bonatti in \cite{bohnet_partially_2014} to define the local unstable projection for the time-one map $\varphi^1$ (Proposition \ref{prop:unstproj}). Thanks to the trivial holonomy of the subcenter foliation, this allows us to define the local unstable projection on the local weak stable leaves through a subcenter leaf (Lemma \ref{lem:unstprojtrivholo}).
Therefore, by the same techniques mentioned in \cite{bohnet_partially_2014}, we prove that the subcenter leaves trough a local weak stable leaf of some $x$ are contained in the local weak stable leaves through the subcenter leaf of $x$, that is if a subcenter leaf $L$ intersects the local stable leaf of $x$, then we know the Hausdorff distance between $\varphi^n(L)$ and $\mathcal{F}^{\hat{c}}(\varphi^n(x))$ tends to $0$ as $n$ goes to $+\infty$, so we get the result with the help of the local unstable projection and the fact that two points on the same local unstable leaf are uniformly expanded with the action of $\varphi^1$ (Lemma \ref{lem:DeltaUinv}).
This proves completeness of the center distribution (Corollary \ref{cor:compcent}), and thus the dynamical coherence of the diffeomorphism $\varphi^1$ (Proposition \ref{prop:dyncohcent}).

In section \ref{sec:6}, we prove the main theorem, that is the subcenter foliation is complete. 
We use the local unstable projection as well as some dynamical properties of a flow to define the local weak unstable projection (Proposition \ref{prop:weakunstproj}).
Thanks to the trivial holonomy of the subcenter foliation, this allows us to define the local weak unstable projection on the local weak stable leaves through a subcenter leaf (Lemma \ref{lem:weakunstprojtrivholo}).
However, we cannot use the same previous arguments since the time-one map $\varphi^1$ does not necessarily uniformly expand two points on the same weak unstable leaf.
We bypass this issue by using the results of the previous section which gives that the restriction to a stable center leaf of the local weak unstable projection can be expressed dynamically with the flow (Lemma \ref{lem:piwu=phi^t(piu)}).
Since the subcenter foliation is invariant by the flow, we prove that the weak unstable projection of every point in a subcenter leaf intersecting the local stable leaf of a point $x$ is equal to the identity, thanks in particular to the fact that the leaf space $M/\mathcal{F}^{\hat{c}}$ is a compact metric space with the Hausdorff metric.
This proves completeness of the center distribution (Corollary \ref{cor:compsubc}), and thus the dynamical coherence of the flow (Proposition \ref{prop:dyncohsubcen}).

Eventually, we prove the local product structure between the center-unstable and subcenter-stable foliation in the quotient leaf space $M/\mathcal{F}^{\hat{c}}$ (Proposition \ref{prop:locprodstruquotient}).

\section{General results on foliations transverse to flows}
\label{sec:4}

\begin{lemma}
\label{lem:weakfol}
    Let $(\varphi^t)_{t\in \mathbb R}$ a complete smooth flow on a smooth manifold M, induced by a nowhere vanishing vector field $X$.
    Let $\mathcal{F}$ a continuous foliation of $M$ with $C^r$ leaves $(r \geq 1)$, which is invariant under the flow $\varphi^t$ (i.e. for every $t \in \mathbb R$ and $x \in M$, $\varphi^t(\mathcal{F}_x)=\mathcal{F}_{\varphi^t(x)}$), such that $X$ is nowhere tangent to $\mathcal{F}$. \\
    Then there exists a continuous foliation $\mathcal{F}^w$ with $C^r$ leaves tangent to $T\mathcal{F} \oplus \mathbb R X$ whose leaf at $x \in M$ is $\mathcal{F}^w(x):=\bigcup_{t \in \mathbb R}\varphi^t(\mathcal{F}_x)$.
\end{lemma}
\begin{proof}
    Denote by $k\geq 1$ the dimension of $\mathcal{F}$. 
    Let $x \in M$ and consider the restriction $\phi_x$ to $\mathbb R\times \mathcal{F}(x)$ of the map 
    $$\phi: \left \{ \begin{array}{ccl}
         \mathbb R\times M& \to & M  \\
         (t,y)&  \mapsto & \varphi^t(y)
    \end{array} \right. \;.$$
    Since $\mathcal{F}$ is invariant under $(\varphi^t)$ and $X$ is nowhere tangent to $\mathcal{F}$, $\phi_x$ is a $C^r$ immersion.
    Therefore $\mathcal{F}^w(x)=\phi( \mathbb R\times \mathcal{F}(x))$ is a $C^r$ immersed submanifold of $M$.\\
    We now prove that the family $(\mathcal{F}^w(x))_{x \in M}$ defines a partition of $M$ by showing:
    \begin{claim}
        For every leaf $L, L'$ of $\mathcal{F}$:
    $$\phi(\mathbb R \times L) \cap \phi(\mathbb R \times L') \neq \emptyset  \iff \phi(\mathbb R \times L) = \phi(\mathbb R \times L').$$
    \end{claim}
    \begin{proof}[Proof of the claim]
    The sufficient part is immediate. We prove the necessary condition. 
    Let $t,t' \in \mathbb R$ and $(x,x') \in L \times L'$ such that $\varphi^{t}(x)=\varphi^{t'}(x')$, i.e. $x'=\varphi^{t-t'}(x)$.
    Now let $z\in L'= \mathcal{F}(x')$ and $t_0 \in \mathbb R$. 
    It comes that $$\varphi^{t_0}(z)\in \varphi^{t_0}(\mathcal{F}(\varphi^{t-t'}(x)))=\varphi^{t_0+t-t'}(\mathcal{F}(x))$$
    by flow invariance of $\mathcal{F}$.
    This proves that $\phi(\mathbb R \times L') \subset \phi(\mathbb R \times L)$.
    By symmetry we conclude.
    \end{proof}
    Denote by $\mathcal{F}^w$ the partition of $M$ by the sets $(\mathcal{F}^w(x))$. We eventually prove that $\mathcal{F}^w$ defines a continuous foliation on $M$.
    Fix $x \in M$ and consider a $C^0$-foliated chart $(U, \psi)$ for $\mathcal{F}$ centered at $x$, i.e. there exist a homeomorphism $\psi:U\to V_1 \times V_2$, where $V_1$ and $V_2$ are open subsets of $\mathbb R^k$ and $\mathbb R^{n-k}$ respectively containing $0$, and $\mathcal{F}|_U$ is sent to the horizontal foliation $V_1\times \{*\}$.\\
    Let $T \subset U$ a small smooth transversal to $\mathbb R X \oplus T\mathcal{F}$ at $x$. Let $\epsilon_0>0$ small enough so that $S:=\bigcup_{|t|<\epsilon_0} \phi^t(T)$ is a smooth transversal to $\mathcal{F}$ diffeomorphic to $\left]-\epsilon_0, \epsilon_0 \right[ \times T$.
    By restricting $U$ if necessary, we define a homeomorphism
    $$\psi'= (\text{Id}, h_0^{\psi(S),V_2})\circ\psi:U \longrightarrow V_1 \times \psi(S)=V_1 \times \left]-\epsilon_0, \epsilon_0 \right[ \times T$$ 
    where $h_0^{\psi(S),V_2}$ is the holonomy of the foliation $\psi_*(\mathcal{F|}_U)$ at $0$ between the transversals $\{0\} \times V_2 \subset \mathbb R^k \times \mathbb R^{n-k}$ and $\psi(S)$. 
    The chart $\psi'$ sends the leaves of $\mathcal{F}$ to $\{V_1\} \times \{*\}\times \{*\}$.
    Consider now, for $\epsilon<\epsilon_0$, the map
    $$\eta: \left \{ \begin{array}{ccl}
         V_1  \times \left] -\epsilon, \epsilon\right[  \times T& \to & M  \\
         (y,t,z)&  \mapsto & \varphi^t(\psi'^{-1}(y,0,z))
         \end{array} \right. \;.$$
    By taking $\epsilon$ small and restricting $V_1$ and $T$ if necessary, we can assume that $\eta$ is a homeomorphism from a neighborhood of $0 \in \mathbb R^n$ to a neighborhood of $x \in M$.
    Moreover, $\eta$ sends the foliation $ \{*\} \times \left] -\epsilon, \epsilon\right[ \times \{*\}$ to $\Phi$ and the foliation $V_1 \times \{*\}\times \{*\}$ to $\mathcal{F}$ by definition of $\psi'$ and since $\mathcal{F}$ is invariant by $(\varphi^t)$. As a consequence $\eta^{-1}$ sends the elements of the partition $\mathcal{F}^w$ to $V_1 \times \left]-\epsilon, \epsilon \right[   \times \{*\}$.\\
    This proves that $\mathcal{F}^w$ is a continuous foliation, and $\eta$ is also $C^0$ adapted to $\mathcal{F}$ and $\Phi$.
    Moreover, since the restriction of $\psi$ to a plaque of $\mathcal{F}$ is $C^r$, the restriction of $\eta$ to a plaque of $\mathcal{F}$ or $\mathcal{F}^w$ is also $C^r$, while its restriction to a plaque of $\Phi$ is $C^\infty$.
\end{proof}
\begin{remark}
\leavevmode
\begin{itemize}
    \item The proof shows that there exists a continuous atlas compatible with the three foliations $\mathcal{F}$, $\Phi$ and $\mathcal{F}^w$ simultaneously and that its restriction to each leaf of these foliations is $C^r$.
    \item If $\mathcal{F}$ is a $C^r$ foliation ($r\geq 1$), then the previous proof shows that $\mathcal{F}^w$ is also a $C^r$ foliation.
\end{itemize}
\end{remark}
If $\mathcal{F}$ is a continuous foliation with $C^1$ leaves on a smooth riemanian manifold $M$ whose metric is fixed, we denote for $x \in M$ and $\delta>0$ by $\mathcal{F}_\delta(x)$ the ball in the leaf $\mathcal{F}(x)$ centered at $x$ and of radius $\delta$ for the induced metric on $\mathcal{F}(x)$.
\begin{lemma}
\label{lem:equivmetricweak}
    Let $(\varphi^t)$ a smooth flow on a smooth compact manifold M, induced by a nowhere vanishing vector field $X$.
    Let $\mathcal{F}$ a continuous foliation of $M$ with $C^r$ leaves $(r \geq 1)$, which is invariant under the flow $\varphi^t$ and such that $X$ is nowhere tangent to $\mathcal{F}$. \\
    Then there exist positive constants $C_1<1$, $C_2>1$ and $\delta_0$ such that for every $\delta<\delta_0$ and $x \in M$, 
    $$\bigcup_{z \in  \Phi_{C_1\delta}(x)}\mathcal{F}_{C_1\delta}(z)\subset \mathcal{F}^{w}_{\delta}(x)\subset \bigcup_{z \in  \Phi_{C_2\delta}(x)}\mathcal{F}_{C_2\delta}(z)$$
    and
    $$\bigcup_{z \in \mathcal{F}_{C_1\delta}(x) }\Phi_{C_1\delta}(z)\subset \mathcal{F}^{w}_{\delta}(x)\subset \bigcup_{z \in \mathcal{F}_{C_2\delta}(x) }\Phi_{C_2\delta}(z).$$
    
\end{lemma}
\begin{proof}
    Consider a continuous atlas of charts $(U_i,\psi_i)_i$ for $\mathcal{F}^w$ given by the previous lemma, centered at some points $x_i$. 
    Write $\psi_i:U_i \to V_i \times \left]-\epsilon_i, \epsilon_i\right[ \times T_i$ where $V_i$ and $T_i$ are open subsets of $\mathbb R^k$ and $\mathbb R^{n-k-1}$ respectively, containing $0$. The restriction of $\psi_i$ to a plaque of $\Phi$ is a $C^1$ diffeomorphism onto some $\{y_0\} \times \left]-\epsilon_i \times , \epsilon_i\right[ \times \{z_0\}$, its restriction to a plaque of $\mathcal{F}$ is a $C^1$ diffeomorphism onto some $V_i \times \{t_0\}\times \{z_0\}$ and its restriction to a plaque of $\mathcal{F}^w$ is a $C^1$ diffeomorphism onto some $V_i \times \left]-\epsilon_i \times , \epsilon_i\right[ \times \{z_0\}$.
    We define two smooth metrics on $V_i \times \left]-\epsilon_i, \epsilon_i\right[ $. 
    The first metric $\|\cdot\|^i_1$ is obtained by pushing forward to $V_i \times \left]-\epsilon_i, \epsilon_i\right[ $, thanks to $\psi_i$, the metric on the open subset $\psi_i^{-1}(V_i \times \left]-\epsilon_i, \epsilon_i\right[ \times \{0\})$ of $\mathcal{F}^w(x_i)$.
    As for the second metric $\|\cdot\|^i_2$, by construction $\psi_i$ is a $C^1$ diffeomorphism between $\left]-\epsilon_i, \epsilon_i\right[ $ and an open neighborhood of $x_i$ in $\Phi(x_i)$, and also between $V_i$ and an open neighborhood of $x_i$ in $\mathcal{F}(x_i)$.
    This allows to define a metric on $\left]-\epsilon_i, \epsilon_i\right[$ and on $V_i$ and therefore on $V_i \times \left]-\epsilon_i, \epsilon_i\right[ $ by product.
    These two metrics can naturally be defined on $V_i \times \left]-\epsilon_i, \epsilon_i\right[ \times \{z_i\}$ for every $z_i \in T_i$.
    Also, we can consider a smaller foliated chart $(\psi_i',U_i')$ centered at $x_i$ whose closure is compact. Therefore we can assume $U_i'=U_i$. 
    By compactness of $\overline{U_i}$, continuity of $T\mathcal{F}$ and equivalence of metrics, there exists $c_i>0$ such that for every $z_i \in T_i$ and $v\in T(V_i \times \left]-\epsilon_i, \epsilon_i\right[ )$:
    $$\frac{1}{c_i}\|v\|_{2,z_i}^i \leq \|v\|_{1,z_i}^i\leq c_i\|v\|_{2,z_i}^i. $$
    It comes that there exist positive constants $C_{i,1}<1, C_{i,2}>1$ and $\delta_i>0$ such that for every $x \in U_i$ and every $\delta<\delta_i$:
    $$\bigcup_{z \in  \Phi_{C_{i,1}\delta}(x)}\mathcal{F}_{C_{i,1}\delta}(z)\subset \mathcal{F}^{w}_{\delta}(x)\subset \bigcup_{z \in  \Phi_{C_{i,2}\delta}(x)}\mathcal{F}_{C_{i,2}\delta}(z)$$
    and 
    $$\bigcup_{z \in  \mathcal{F}_{C_{i,1}\delta}(x)}\Phi_{C_{i,1}\delta}(z)\subset \mathcal{F}^{w}_{\delta}(x)\subset \bigcup_{z \in  \mathcal{F}_{C_{i,2}\delta}(x)}\Phi_{C_{i,2}\delta}(z).$$
    Since $M$ is compact, we can consider a finite number of charts $(U_i, \psi_i)$, so $\delta_0=\min_i(\delta_i)$, $C_1=\min_i(C_{i,1})$ and $C_2=\max_i(C_{i,2})$ are well-defined and satisfy the conclusion of the Lemma.
\end{proof}
We can use the same reasoning to prove the following Lemma.
\begin{lemma}
\label{lem:equivmetricflo}
    Let $(\varphi^t)$ a smooth flow on a smooth compact manifold M, induced by a nowhere vanishing vector field $X$.\\
    Then there exist positive constants $C_1<1$, $C_2>1$ and $\delta_0$ such that for every $\delta<\delta_0$ and $x \in M$, 
    $$\bigcup_{t \in  \left]-C_1\delta, C_1\delta \right[}\varphi^t(x)\subset \Phi_\delta(x)\subset \bigcup_{t \in  \left]-C_2\delta, C_2\delta \right[}\varphi^t(x).$$
\end{lemma}
\begin{proof}
    The proof is analogous to the previous one, except that we use a $C^\infty$ flow box $(U, \psi)$ centered at $x\in M$ instead of a foliation chart ; the first metric is the push-forward of the metric in the plaque of $x$, while the second metric is the standard flat euclidean metric on $\left]-\epsilon, \epsilon \right[$
\end{proof}

The previous shows that we can define $\mathcal{F}^c:=\mathcal{F}^{w\hat{c}}$, $\mathcal{F}^{ws}$ and $\mathcal{F}^{wu}$. 
We consider common constants $C_1, C_2$ and $\delta_0$ for the foliations $\mathcal{F}^s$, $\mathcal{F}^{\hat{c}}$ and $\mathcal{F}^u$ given by Lemma \ref{lem:equivmetricweak}.

\begin{lemma}
\label{lem:transvfol}
    Let $M$ a smooth compact manifold, and $\mathcal{F}$, $\mathcal{F}'$ two continuous foliations with $C^1$-leaves on $M$ which are transverse (i.e. for every $x\in M$, $T_x\mathcal{F}_1 \cap T_x\mathcal{F}_2=\{0\}$).\\
    Then there exists $\mu>0$ such that for every $x \in M$, $\mathcal{F}_{\mu}(x)\cap \mathcal{F}_{\mu}'(x)=\{x\}$.
\end{lemma}
\begin{proof}
    We first show that for every $x \in M$, there exists an open neighborhood $U_x$ of $x$ in $M$ and $\mu_x>0$ such that for every $y \in U_x$, $\mathcal{F}_{\mu_x}(y)\cap \mathcal{F}_{\mu_x}'(y)=\{y\}$.
    Since the properties to prove are local, we can assume $M=\mathbb R^n$, $x=0$, $T_0\mathcal{F}=\mathbb R^{k}\times \{0\}^{k'}\times \{0\}^{n-k-k'}$, $T_0\mathcal{F}'= \{0\}^{k} \times \mathbb R^{k'}\times \{0\}^{n-k-k'}$ (where $k=\dim(\mathcal{F}), k'=\dim(\mathcal{F})$).
    Since $T\mathcal{F}$ and $T\mathcal{F}'$ are continuous bundle, for every $\epsilon>0$, there exists a neighborhood $V$ of $0$ such that for every $y \in V$, the angle between $T_y\mathcal{F}$ and $\mathbb R^{k}\times \{0\}^{k'}\times \{0\}^{n-k-k'}$ and the angle between $T_y\mathcal{F}'$ and $ \{0\}^{k} \times \mathbb R^{k'}\times \{0\}^{n-k-k'}$ are less than $\epsilon$.
    For such $\epsilon>0$ fixed, there exists $\mu_\epsilon>0$ and a neighborhood $U_\epsilon$ of $0$ such that for every $y \in U$, $\mathcal{F}_{\mu}(y),\mathcal{F}'_{\mu}(y) \subset V$.
    If we choose initially $\epsilon$ small enough, then for every $y \in U_\epsilon$, $\mathcal{F}_{\mu}(y) \cap \mathcal{F}'_{\mu}(y) =\{y\}$.\\
    Since $M$ is compact, we can recover it by a finite number of such open sets $(U_i)_{i}$. 
    If we take $\mu=\min_i \mu_i>0$, then we get the result.
\end{proof}
We will let $\mu_*$ small enough given by the previous corollary for any encountered couple of continuous transverse foliations (which are in finite number).
\begin{lemma}
\label{lem:transvcompafol}
     Let $M$ a smooth compact manifold, and $\mathcal{F}$, $\mathcal{F}'$ two continuous foliations with $C^1$-leaves on $M$ which are transverse and such that the leaves of $\mathcal{F}'$ are compact.\\
    Then there exists $\mu>0$ such that for every $x \in M$, $\mathcal{F}_{\mu}(x)\cap \mathcal{F}'(x)=\{x\}$.
\end{lemma}
\begin{proof}
    Let $\mu$ small enough so that for every two points $w_1,w_2$ on the same leaf of $\mathcal{F}'$ with $d(w_1,w_2)<\mu$, then $d'(w_1,w_2)<2d(w_1,w_2)$ (which exists by Remark \ref{rem:fol}).
    Now let $x\in M$ and $y\in \mathcal{F}_{\mu}(x)\cap\mathcal{F}'(x)$. Then $d(x,y)<\mu$, so $d'(x,y)<2\mu$ by the above.
    By the previous Lemma, if $\mu$ is sufficiently small (without depending on $x,y$), then $y=x$ which proves the result.
\end{proof}
We end this section by a fundamental result on triples of transverse foliations, which is proven in \cite{bohnet_partially_2014}:
\begin{lemma}
\label{lem:bonhettriple}
    Let $\mathcal{F}_1$, $\mathcal{F}_2$, $\mathcal{F}_3$ three continuous foliations with $C^1$ leaves of a smooth compact $M$, tangent respectively to continuous distributions $E_1$, $E_2$, $E_3$ such that $TM=E_1 \oplus E_2 \oplus E_3$.\\
    Then there exist positive constants $\delta_0,C<1$ such that for every $\delta<\delta_0$ and $x,y \in M$ with $d(x,y)<C\delta$, the set 
    $$\mathcal{F}_{1,\delta}(y)\cap\bigcup_{z \in \mathcal{F}_{2,\delta}(x)}\mathcal{F}_{3, \delta}(z)$$
    is non-empty.
\end{lemma}

\section{Completeness of the center foliation}
\label{sec:5}

From now, we assume $(\varphi^t)$ is a partially hyperbolic flow with a flow invariant subcenter foliation $\mathcal{F}^{\hat{c}}$ with $C^1$ leaves.
All of the following results are true if we switch the unstable and stable roles, by considering the reverse flow.
\subsection{Unstable projection}

\begin{proposition}
\label{prop:unstproj}
    Let $(\varphi^t)$ a smooth partially hyperbolic flow with a flow invariant subcenter foliation $\mathcal{F}^{\hat{c}}$ with $C^1$-leaves. Let $0<C,\delta_0<1$ given by Lemma \ref{lem:bonhettriple} for $\mathcal{F}_1=\mathcal{F}^u$, $\mathcal{F}_2=\mathcal{F}^{\hat{c}}$ and $\mathcal{F}_3=\mathcal{F}^{ws}$.\\
    Then there exists $\mu_0<\delta_0$ such that:
    \begin{enumerate}[label=(\roman*)]
        \item For every $x,y \in M$ with $d(x,y)<\mu_0$, the set 
        $$\mathcal{F}^{u}_{\mu_0}(y)\cap\bigcup_{z \in \mathcal{F}^{\hat{c}}_{\mu_0}(x)}\mathcal{F}^{ws}_{ \mu_0}(z)$$
        contains at most one point;
        \item Furthermore, if $d(x,y)<C\mu$ where $\mu<\mu_0$, then this set consists of a single element $\pi^{u}(x,y)$ which is equal to the singleton
        $$\mathcal{F}^{u}_{\mu}(y)\cap\bigcup_{z \in \mathcal{F}^{\hat{c}}_{\mu}(x)}\mathcal{F}^{ws}_{\mu}(z),$$
        and the map $(x,y)\mapsto\pi^u(x,y)$ is continuous.
    \end{enumerate}
\end{proposition}
\begin{proof}
    We apply Lemma \ref{lem:bonhettriple} to the time-one map $\varphi^1$ which is partially hyperbolic and whose stable, center and unstable foliations are $\mathcal{F}^s$, $\mathcal{F}^c$ and $\mathcal{F}^u$ respectively. This gives $\mu_0>0$ such that for every $x,y \in M$ with $d(x,y)<\mu_0$, 
    $$\mathcal{F}^{u}_{\mu_0}(y)\cap\bigcup_{z \in \mathcal{F}^{c}_{\mu_0}(x)}\mathcal{F}^{s}_{ \mu_0}(z)$$
    contains at most one point.
    If $\mu<\min(\delta_0\frac{C_1}{C_2}, \mu_0\frac{C_1}{C_2})<\min(\delta_0, \mu_0)$, then by Lemma \ref{lem:equivmetricweak},
    $$\bigcup_{z \in \mathcal{F}^{\hat{c}}_{\mu}(x)}\mathcal{F}^{ws}_{ \mu}(z) \subset \bigcup_{z \in \mathcal{F}^{\hat{c}}_{\mu}(x)}\bigcup_{z' \in \Phi_{C_2\mu}(z)}\mathcal{F}^{s}_{ C_2\mu}(z')\subset \bigcup_{z \in \mathcal{F}^{c}_{\frac{C_2}{C_1}\mu}(x)}\mathcal{F}^{s}_{ C_2\mu}(z) \subset \bigcup_{z \in \mathcal{F}^{c}_{\mu_0}(x)}\mathcal{F}^{s}_{ \mu_0}(z).$$
    Therefore, if $d(x,y)< \mu$, 
    $$\mathcal{F}^{u}_{\mu}(y)\cap\bigcup_{z \in \mathcal{F}^{\hat{c}}_{\mu}(x)}\mathcal{F}^{ws}_{ \mu}(z) \subset \mathcal{F}^{u}_{\mu_0}(y)\cap\bigcup_{z \in \mathcal{F}^{c}_{\mu_0}(x)}\mathcal{F}^{s}_{ \mu_0}(z)$$
    contains at most one point.\\
    The last point comes from Lemma \ref{lem:bonhettriple}. The continuity of $\pi^u$ can be checked thanks to continuity of all involved foliations.
    \end{proof}
\begin{remark}
\leavevmode
    \begin{itemize}
        \item The last point of the previous Lemma can be used in the following way. 
    Take $x,y \in M$ such that $d(x,y)<\frac{C\mu_0}{2}$. Then $d(x,y)<C\cdot \frac{2d(x,y)}{C}$ and $\frac{2d(x,y)}{C}<\mu_0$ so $\pi^u(x,y)$ is equal to the singleton
    $$\mathcal{F}^{u}_{\frac{2d(x,y)}{C}}(y)\cap\bigcup_{z \in \mathcal{F}^{\hat{c}}_{\frac{2d(x,y)}{C}}(x)}\mathcal{F}^{ws}_{\frac{2d(x,y)}{C}}(z).$$
        \item For $x,y \in M$ such that $d(x,y)<C\mu$ where $\mu < \mu_0$, it comes:
        $$y\in \bigcup_{z \in \mathcal{F}^{\hat{c}}_{\mu}(x)}\mathcal{F}^{ws}_{\mu}(z) \iff \pi^u(x,y)=y.$$
    \end{itemize}
\end{remark}

The result of Proposition \ref{prop:unstproj} allows us to prove dynamical coherence of $\varphi^1$ from completeness of the center foliation:

\begin{proposition}
\label{prop:dyncohcent}
    Let $(\varphi^t)$ a smooth partially hyperbolic flow with a flow invariant subcenter foliation $\mathcal{F}^{\hat{c}}$ with $C^1$-leaves. Assume that the center foliation $\mathcal{F}^c$ is complete.\\
    Then there exists a continuous foliation $\mathcal{F}^{cs}$ with $C^1$ leaves tangent to $E^{cs}$.
    Its leaf at $x \in M$ is $\mathcal{F}^{cs}(x):=\bigcup_{z \in \mathcal{F}^{\hat{c}}(x)}\mathcal{F}^{ws}(z)$.
    In particular, for every $t \in \mathbb R$ and $x \in M$, $\varphi^t(\mathcal{F}^{cs}(x))=\mathcal{F}^{cs}(x)=\mathcal{F}^{cs}(\varphi^t(x))$.
    \end{proposition}
\begin{proof}
    Since $\varphi^1$ is a partially hyperbolic diffeomorphism with center bundle given by $E^c$, the results of \cite{hirsch_invariant_1977} (Theorem 6.1) give that there exits $\delta>0$ such that for every $x \in M$, 
    $\mathcal{F}_\delta^{cs}(x):=\bigcup_{z \in \mathcal{F}_\delta^{c}(x)}\mathcal{F}_\delta^{s}(z)$ is an immersed $C^1$ submanifold of $M$, tangent to $E^{cs}$.
    Therefore, $\mathcal{F}^{cs}(x)=\bigcup_{z \in \mathcal{F}^{\hat{c}}(x)}\mathcal{F}^{ws}(z)=\bigcup_{z \in \mathcal{F}^{c}(x)}\mathcal{F}^{s}(z)$ is an immersed $C^1$ submanifold of $M$ tangent to $E^{cs}$.
    We now construct a chart at $x \in M$ adapted to $\mathcal{F}^{cs}$.
    Let $\gamma,c>0$ small enough so that the map
    $$\eta:\left \{\begin{array}{ccl}
         \mathcal{F}_{c\gamma}^{cs}(x)\times \mathcal{F}_{c\gamma}^{u}(x) & \to & M \\
         (y,z) & \to & \mathcal{F}_{\gamma}^{u}(y)\cap \mathcal{F}_{\gamma}^{cs}(z)
    \end{array} \right.$$
    is a well-defined homeomorphism onto its image by Proposition \ref{prop:unstproj}. 
    By identifying $\mathcal{F}_{c\gamma}^{cs}(x)$ and $\mathcal{F}_{c\gamma}^{u}(x)$ to open connected subsets of $\mathbb R^{d_c+d_s}$ and $\mathbb R^{d_u}$ respectively, we see that $\eta^{-1}$ is an adapted foliation chart for $\mathcal{F}^{cs}$ whose restriction at plaques of $\mathcal{F}^{cs}$ is $C^1$.
\end{proof}
We proceed in the case of a flow invariant subcenter foliation with $C^1$ leaves and show that the unstable projection defined above is $\varphi^1$-invariant:

\begin{lemma}
\label{lem:unstprojinv}
    Let $(\varphi^t)$ a smooth partially hyperbolic flow with a flow invariant subcenter foliation $\mathcal{F}^{\hat{c}}$ with $C^1$-leaves.\\
    Then there exist $\mu_0'>0$ such that 
    for every $x,y\in M$ with $d(x,y)<\mu_0'$ and $d(\varphi^1(x),\varphi^1(y))<\mu_0'$, $\pi^{u}(x,y)$ and $\pi^{u}\left(\varphi^1(x),\varphi^1(y)\right)$ are well defined and satisfy 
    $$\pi^{u}\left(\varphi^1(x), \varphi^1(y)\right)=\varphi^1\left(\pi^{u}(x,y)\right).$$
\end{lemma}
\begin{proof}
    Let $x,y \in M$ such that $d(x,y)<C\mu$ and $d(\varphi^1(x),\varphi^1(y))<C\mu$, with $\mu<\mu_0$ and $\mu_0$ is given by the previous Lemma. 
    It comes that $p:=\pi^{u}(x,y)$ and $\pi^{u}\left(\varphi^1(x),\varphi^1(y)\right)$ are well defined so we just need to prove, by taking $\mu$ smaller if necessary, that $\varphi^1(p)\in\mathcal{F}^{u}_{\mu_0}(\varphi^1(y))\cap\bigcup_{z' \in \mathcal{F}^{\hat{c}}_{\mu_0}(\varphi^1(x))}\mathcal{F}^{ws}_{ \mu_0}(z')$.
    Let $z \in \mathcal{F}^{\hat{c}}_\mu(x)$ such that $p\in \mathcal{F}^{ws}_\mu(z)$. 
    It comes that $$d^u(y,p)<\mu, \;d^{ws}(p,z)<\mu, \;d^{\hat{c}}(z,x)<\mu.$$ 
    As a result, if we let $\Lambda:=\sup_{w \in M}\|d_w\varphi^1\|>1$ which is finite since $M$ is compact, $$d^u(\varphi^1(y),\varphi^1(p))<\Lambda\mu, \;d^{ws}(\varphi^1(p),\varphi^1(z))<\Lambda\mu, \;d^{\hat{c}}(\varphi^1(z),\varphi^1(x))<\Lambda\mu.$$
    Therefore, if $\mu<\frac{\mu_0}{\Lambda}$, since all foliations are $\varphi^1$-invariant, 
    $$\varphi^1(p)\in\mathcal{F}^{u}_{\mu_0}(\varphi^1(y))\cap\bigcup_{z' \in \mathcal{F}^{\hat{c}}_{\mu_0}(\varphi^1(x))}\mathcal{F}^{ws}_{ \mu_0}(z')$$
    which concludes.
\end{proof}

In case $\mathcal{F}^{\hat{c}}$ has trivial holonomy, then:

\begin{lemma}
\label{lem:unstprojtrivholo}
    Let $(\varphi^t)$ a smooth partially hyperbolic flow with a flow invariant subcenter foliation $\mathcal{F}^{\hat{c}}$ with $C^1$-leaves and trivial holonomy.\\
    Then there exists $\mu_1>0$ such that 
    for every $x\in M$ and subcenter leaf $L^{\hat{c}}$ with $d(x,L^{\hat{c}})<C\mu_1$, the set 
        $$\mathcal{F}^{u}_{\mu_1}(x)\cap\bigcup_{z \in L^{\hat{c}}}\mathcal{F}^{ws}_{ \mu_1}(z)$$
        consists of a single point.\\
    Moreover, for every $x \in M$, $\mathcal{F}^{ws}_{2\mu_1}(x)\cap\mathcal{F}^{\hat{c}}(x)=\{x\}$.
\end{lemma}
\begin{proof}
    The proof is similar to that of Lemma 4.8 of \cite{bohnet_partially_2014}.\\
    Let $x \in M$ and assume $d(x,L^{\hat{c}})<C\mu$ with $\mu< \mu_0$ so that $\mathcal{F}^{u}_{\mu}(x)\cap\bigcup_{z \in L^{\hat{c}}}\mathcal{F}^{ws}_{ \mu}(z)$
    is at least non-empty by Proposition \ref{prop:unstproj} (since $L^{\hat{c}}$ is compact, then $d(x,L^{\hat{c}})=d(x,y)$ for some $y \in L^{\hat{c}}$) .\\
    We prove the uniqueness. Let $\mu<\frac{\mu_0}{4}$ and consider two points $w_1,w_2$ in $\mathcal{F}^{u}_{\mu}(x)\cap\bigcup_{z \in L^{\hat{c}}}\mathcal{F}^{ws}_{ \mu}(z)$. 
    For $i\in \{1,2\}$, let $z_i\in L^{\hat{c}}$ such that $w_i \in \mathcal{F}^{ws}_{\mu}(z_i)$.
    Then $d(z_1, z_2)< d(z_1,w_1)+d(w_1,x)+d(x,w_2)+d(w_2,z_2)<4\mu$ and $d(x,z_1), d(x,z_2)<2\mu<\mu_0$.
    By remark \ref{rem:fol}, we can take $\mu$ small enough (depending only on $\mathcal{F}^{\hat{c}}$) so that $d(z_1, z_2)<4\mu$ implies $d^{\hat{c}}(z_1,z_2)<5\mu$.
    If we assume furthermore that $\mu<\frac{\mu_0}{5}$, then both $w_1$ and $w_2$ belong to $\mathcal{F}^{u}_{\mu_0}(x)\cap\bigcup_{z' \in \mathcal{F}^{\hat{c}}_{\mu_0}(z_1)}\mathcal{F}^{ws}_{ \mu_0}(z')$ which contains at most one point by Proposition \ref{prop:unstproj} (since $d(x,z_1)< \mu_0$), so $w_1=w_2$ and the proof of the first point is complete.\\
    The proof of the second point is a consequence of Corollary \ref{lem:transvcompafol}.
    \end{proof}
In particular, the previous corollary implies that for every $x\in M$, every point $y$ in $\bigcup_{z \in \mathcal{F}^{\hat{c}}(x)}\mathcal{F}^{ws}_{ \mu_1}(z)$ belongs to some $\mathcal{F}^{ws}_{ \mu_1}(z)$ for a unique $z\in \mathcal{F}^{\hat{c}}(x)$.
Also, it can be checked that the map $(x,L^{\hat{c}}) \mapsto \mathcal{F}^{u}_{\mu_1}(x)\cap\bigcup_{z \in L^{\hat{c}}}\mathcal{F}^{ws}_{ \mu_1}(z)$ is continuous.

\begin{lemma}
\label{lem:Piu}
    Let $(\varphi^t)$ a smooth partially hyperbolic flow with a flow invariant subcenter foliation $\mathcal{F}^{\hat{c}}$ with $C^1$-leaves and trivial holonomy.\\
    Then there exist $\mu_2>0$ such that 
    for every $x,y\in M$ with $d(x,y)<\mu_2$ and every $w \in \mathcal{F}^{\hat{c}}(y)$, the set 
        $$\mathcal{F}^{u}_{\mu_1}(w)\cap\bigcup_{z \in \mathcal{F}^{\hat{c}}(x)}\mathcal{F}^{ws}_{ \mu_1}(z)$$
        consists of a single element denoted by $\Pi^u_{x,y}(w)$.\\
    Moreover, the map $w \mapsto \Pi^u_{x,y}(w)$ is continuous.
\end{lemma}
\begin{proof}
    Since the leaf space $N:=M/\mathcal{F^{\hat{c}}}$ is a compact metric space with the Hausdorff distance and because the projection map $M \to N$ is continuous and thus uniformly continuous on $M$, there exists $\mu_2< \mu_1$ such that for every $x,y \in M$ with $d(x,y)<\mu_2$, $d(\mathcal{F}^{\hat{c}}(x),\mathcal{F}^{\hat{c}}(y))<C\mu_1$.
    Therefore if $d(x,y)<\mu_2$, then for every $w \in \mathcal{F}^{\hat{c}}(y)$, $d(w,\mathcal{F}^{\hat{c}}(x))<C\mu_1$ so the previous corollary concludes.
    The continuity of $\Pi^u_{x,y}$ comes from the continuity of $\pi^u$.
\end{proof}

\begin{lemma}
\label{lem:Piuinv}
    Let $(\varphi^t)$ a smooth partially hyperbolic flow with a flow invariant subcenter foliation $\mathcal{F}^{\hat{c}}$ with $C^1$-leaves and trivial holonomy.\\
    Then there exist $\mu_3>0$ such that 
    for every $x,y\in M$ with $d(x,y)<\mu_3$ and $d(\varphi^1(x),\varphi^1(y))<\mu_3$, for every $w \in \mathcal{F}^{\hat{c}}(y)$, $\Pi^{u}_{x,y}(w)$ and $\Pi^{u}_{\varphi^1(x),\varphi^1(y)}(\varphi^1(w))$ are well defined and satisfy 
    $$\Pi^{u}_{\varphi^1(x),\varphi^1(y)}(\varphi^1(w))=\varphi^1\left(\Pi^{u}_{x,y}(w) \right).$$
\end{lemma}
\begin{proof}
    This is immediate by the previous corollary and Lemma \ref{lem:unstprojinv}.
\end{proof}

\subsection{Unstable projection distance and completeness of $E^c\oplus E^s$}

For $x,y\in M$ with $d(x,y)<\mu_2$, we denote by 
$$\Delta^{u}(x,y):=\sup_{w \in \mathcal{F}^{\hat{c}}(y)}d_u \left(w, \Pi^{u}_{x,y}(w) \right).$$
Remark that by construction $\Delta^u(x,y)<\mu_2$.
\begin{lemma}
\label{lem:DeltaUcarac}
    Let $x, y \in M$ such that $d(x,y)<\mu_2$.\\
    Then $$\Delta^u(x,y)=0 \iff \mathcal{F}^{\hat{c}}(y) \subset\bigcup_{z \in \mathcal{F}^{\hat{c}}(x)}\mathcal{F}_{\mu_1}^{ws}(z).$$
\end{lemma}
\begin{proof}
    If $\Delta^u(x,y)=0$, then for every $w \in \mathcal{F}^{\hat{c}}(y)$, $w=\Pi^{u}_{x,y}(w)\in \mathcal{F}^{u}_{\mu_1}(w)\cap\bigcup_{z \in \mathcal{F}^{\hat{c}}(x)}\mathcal{F}^{ws}_{ \mu_1}(z)$.
    Conversely, for every $w \in \mathcal{F}^{\hat{c}}(y)$, $w \in \mathcal{F}^{u}_{\mu_1}(w)\cap\bigcup_{z \in \mathcal{F}^{\hat{c}}(x)}\mathcal{F}^{ws}_{ \mu_1}(z)=\left\{\Pi^{u}_{x,y}(w)\right\}$. 
    Thus $\Delta^u(x,y)=0$.
\end{proof}
\begin{lemma}
\label{lem:DeltaUinv}
     Let $x, y \in M$ such that $d(x,y)<\mu_3$ and $d(\varphi^1(x),\varphi^1(y))<\mu_3$.\\
    Then $$\Delta^u(\varphi^1(x),\varphi^1(y)) \geq \beta\Delta^u(x,y).$$
\end{lemma}
\begin{proof}
    By definition, $$\Delta^{u}(\varphi^1(x),\varphi(y)):=\sup_{w \in \mathcal{F}^{\hat{c}}(y)}d_u \left(\varphi^1(w), \Pi^{u}_{\varphi^1(x),\varphi^1(y)}(\varphi^1(w)) \right).$$
    For every $w \in \mathcal{F}^{\hat{c}}(x)$, it comes by Lemma \ref{lem:Piuinv}:
    $$\Delta^{u}(\varphi^1(x),\varphi(y))\geq d_u \left(\varphi^1(w), \Pi^{u}_{\varphi^1(x),\varphi^1(y)}(\varphi^1(w)) \right)=d_u \left(\varphi^1(w), \varphi^1(\Pi^{u}_{x,y}(w)) \right)\geq \beta d_u \left(w, \Pi^{u}_{x,y}(w) \right).$$
    We conclude by taking the supremum over all $w \in \mathcal{F}^{\hat{c}}(x)$ in the right hand side.
\end{proof}

\begin{corollary}
\label{cor:Fçinclus..}
    Let $x\in M$ and $y\in \mathcal{F}_{\mu_3}^s(x)$.\\
    Then $\mathcal{F}^{\hat{c}}(y) \subset\bigcup_{z \in \mathcal{F}^{\hat{c}}(x)}\mathcal{F}_{\mu_1}^{ws}(z).$
\end{corollary}
\begin{proof}
    By Lemma \ref{lem:DeltaUinv}, for every $n \in \mathbb N$, 
    $$\mu_2>\Delta^u(\varphi^n(x),\varphi^n(y)) \geq \beta^n\Delta^u(x,y).$$
    Therefore, $\Delta^u(x,y)=0$ since $\beta>1$, which concludes by Lemma \ref{lem:DeltaUcarac}.
\end{proof}

\begin{remark}
    In fact by Proposition \ref{prop:unstproj} and the previous Corollary, if $\mu$ is sufficiently small, then for every $x\in M$ and $y\in \mathcal{F}_\mu^s(x)$, $\mathcal{F}^{\hat{c}}(y) \subset\bigcup_{z \in \mathcal{F}^{\hat{c}}(x)}\mathcal{F}_{\mu_1/2}^{ws}(z).$
    Therefore by Lemma \ref{lem:equivmetricweak}, there exists $\mu_4>0$ such that for every $x\in M$ and $y\in \mathcal{F}^{ws}_{\mu_4}(x)$, then $\mathcal{F}^{\hat{c}}(y) \subset\bigcup_{z \in \mathcal{F}^{\hat{c}}(x)}\mathcal{F}_{\mu_1}^{ws}(z).$
\end{remark}

\begin{corollary}
\label{cor:Fcs=FçxFsloc}
    For every $x\in M$, the map
    $$\left \{\begin{array}{ccl}
        \mathcal{F}^{ws}_{\mu_4}(x) \times \mathcal{F}^{\hat{c}}(x)&\to & M\\
         (y,z) &\mapsto & \mathcal{F}^{\hat{c}}(y)\cap \mathcal{F}^{ws}_{\mu_1}(z)
    \end{array} \right.$$
    is a well-defined homeomorphism onto its image.
\end{corollary}
\begin{proof}
    Let $y \in \mathcal{F}^{ws}_{\mu_4}(x)$ and $z \in \mathcal{F}^{\hat{c}}(x)$. Then $x \in \mathcal{F}^s_{\mu_4}(y)$ so by the previous result, there exists a point $w \in \mathcal{F}^{\hat{c}}(y)$ such that $z\in \mathcal{F}^s_{\mu_1}(w)$ which is unique. Therefore, the above map is well-defined, and it can be checked that it is a homeomorphism onto its image.
\end{proof}
To conclude the proof of completeness of the center distribution, remark that for every $x \in M$, 
$$\bigcup_{z \in \mathcal{F}^{\hat{c}}(x)}\mathcal{F}^{ws}(z)=\bigcup_{z \in \mathcal{F}^{c}(x)}\mathcal{F}^{s}(z) \quad \text{ and } \quad \bigcup_{z \in \mathcal{F}^{s}(x)}\mathcal{F}^{c}(z)=\bigcup_{z \in \mathcal{F}^{ws}(x)}\mathcal{F}^{\hat{c}}(z) .$$
Also, for $x,y \in M$, $\mathcal{F}^{c}(y) \subset\bigcup_{z \in \mathcal{F}^{c}(x)}\mathcal{F}^{s}(z)$ if and only if $\mathcal{F}^{\hat{c}}(y) \subset\bigcup_{z \in \mathcal{F}^{c}(x)}\mathcal{F}^{s}(z)$.

\begin{corollary}
\label{cor:compcent}
    For every $x \in M$, 
    $$\bigcup_{z \in \mathcal{F}^{s}(x)}\mathcal{F}^{c}(z) = \bigcup_{w \in \mathcal{F}^{c}(x)}\mathcal{F}^{s}(w).$$
\end{corollary}
\begin{proof}
    Let $z \in \mathcal{F}^{s}(x)$ and $y \in \mathcal{F}^c(z)$. 
    Then $d(\varphi^n(z), \varphi^n(x)) \xrightarrow[n \to +\infty]{}0$, so for $n$ large enough $d^s(\varphi^n(z), \varphi^n(x)) <\mu_3$. 
    This implies 
    $$\varphi^n(\mathcal{F}^{\hat{c}}(z))=\mathcal{F}^{\hat{c}}(\varphi^n(z)) \subset\bigcup_{w \in \mathcal{F}^{\hat{c}}(x)}\mathcal{F}^{ws}_{\mu_1}(w)$$ by Corollary \ref{cor:Fçinclus..}, so
    $$\mathcal{F}^{\hat{c}}(z) \subset\bigcup_{w \in \mathcal{F}^{\hat{c}}(x)}\mathcal{F}^{ws}(w)$$ 
    Therefore $y \in \mathcal{F}^{c}(z)\subset \bigcup_{w \in \mathcal{F}^{\hat{c}}(x)}\mathcal{F}^{ws}(w)=\bigcup_{w \in \mathcal{F}^{c}(x)}\mathcal{F}^{s}(w)$. \\
    Conversely, let $z \in \mathcal{F}^c(x)$ and $y \in \mathcal{F}^s(z)$. Then $z \in \mathcal{F}^s(y)$ and $x \in \mathcal{F}^c(z)$ so by the previous point $x \in \bigcup_{w \in \mathcal{F}^{c}(y)}\mathcal{F}^{s}(w)$ i.e. there exists $w\in \mathcal{F}^c(y)$ such that $x \in \mathcal{F}^s(w)$. 
    Therefore, $y\in \mathcal{F}^c(w)$ with $w \in \mathcal{F}^s(x)$ which concludes.
\end{proof}

\section{Completeness of the subcenter foliation}
\label{sec:6}

Again, all of the following results are true if we switch the unstable and stable roles, by considering the reverse flow.

\subsection{Weak unstable projection}

\begin{proposition}
\label{prop:weakunstproj}
    Let $(\varphi^t)$ a smooth partially hyperbolic flow with a flow invariant subcenter foliation $\mathcal{F}^{\hat{c}}$ with $C^1$-leaves. Let $0<C',\delta_0'<1$ given by Lemma \ref{lem:bonhettriple} for $\mathcal{F}_1=\mathcal{F}^{wu}$, $\mathcal{F}_2=\mathcal{F}^{\hat{c}}$ and $\mathcal{F}_3=\mathcal{F}^s$.\\
    Then there exist $\nu_0>0$ such that:
    \begin{enumerate}[label=(\roman*)]
        \item For every $x,y \in M$ with $d(x,y)<\nu_0$, the set 
        $$\mathcal{F}^{wu}_{\nu_0}(y)\cap\bigcup_{z \in \mathcal{F}^{\hat{c}}_{\nu_0}(x)}\mathcal{F}^{s}_{ \nu_0}(z)$$
        contains at most one point;
        \item Furthermore, if $d(x,y)<C'\nu$ where $\nu<\nu_0$, then this set consists of a single element $\pi^{wu}(x,y)$ which is equal to the singleton
        $$\mathcal{F}^{wu}_{\nu}(y)\cap\bigcup_{z \in \mathcal{F}^{\hat{c}}_{\nu}(x)}\mathcal{F}^{s}_{\nu}(z),$$
        and the map $(x,y)\mapsto\pi^{wu}(x,y)$ is continuous.
    \end{enumerate}
\end{proposition}
\begin{proof}
    As for the first point, assume on the contrary that for every $\nu>0$, there exist $x,y \in M$ with $d(x,y)<\nu$ and two distinct points $w_1,w_2$ in $\mathcal{F}^{wu}_{\nu}(y)\cap\bigcup_{z \in \mathcal{F}^{\hat{c}}_{\nu}(x)}\mathcal{F}^{s}_{\nu}(z)$.
    Take $\nu<\frac{\delta_0}{C_2}$. By Lemma \ref{lem:equivmetricweak}, 
    $w_1,w_2\in  \mathcal{F}^{wu}_{\nu}(y)\subset \bigcup_{z \in \mathcal{F}^u_{C_2\nu}(y) }\Phi_{C_2\nu}(z).$
    Therefore by Lemma \ref{lem:equivmetricflo}, there exist and $t_1,t_2 \in \left ]-C_2^2\nu, C_2^2\nu \right [$ such that $\varphi^{t_1}(w_1),\varphi^{t_2}(w_2)\in \mathcal{F}^u_{C_2\nu}(y)$.
    Also, for $i \in \{1,2\}$, there exists $z_i \in \mathcal{F}^{\hat{c}}_{\nu}(x)$ such that $w_i\in \mathcal{F}^{s}_{\nu}(z_i)$. 
    By theses lemmas, if $\nu$ is sufficiently small, this implies that for $i \in \{1,2\}$, $\varphi^{t_i}(w_i) \in \mathcal{F}^{ws}_{\mu_0}(z_i)$.
    By Proposition \ref{prop:unstproj}, $\varphi^{t_1}(w_1)=\varphi^{t_2}(w_2)$.
    In particular $w_2$ and $w_1$ both belong to the same orbit.
    Therefore, if $\nu$ is sufficiently small, $z_2\in \mathcal{F}^{ws}_{\mu_*}(z_1)\cap \mathcal{F}^{\hat{c}}_{\mu_*}(z_1)=\{z_1\}$ (see the remark right after Lemma \ref{lem:transvfol}), i.e. $z_1=z_2$.
    This implies, if $\nu$ is sufficiently small, that $w_2\in \Phi_{\mu_*}(w_1)\cap \mathcal{F}^{s}_{\mu_*}(w_1)=\{w_1\}$ which is absurd since $w_1\neq w_2$.\\
    The last point comes from Lemma \ref{lem:bonhettriple}.
    \end{proof}

\begin{remark}
\leavevmode
    \begin{itemize}
        \item The proof given in \cite{bohnet_partially_2014} does not seem to apply for the weak unstable foliation instead of the unstable one. This comes from the fact two points on the same local weak unstable leaf are not necessarily expanded by the action of the flow.
        \item For $x,y \in M$ such that $d(x,y)<C\nu$ where $\nu < \nu_0$, it comes:
        $$y\in \bigcup_{z \in \mathcal{F}^{\hat{c}}_{\mu}(x)}\mathcal{F}^{s}_{\mu}(z) \iff \pi^{wu}(x,y)=y.$$
        
    \end{itemize}
    
\end{remark}

The previous result allows us to prove dynamical coherence of the flow from completeness of the subcenter foliation.

\begin{lemma}
\label{lem:weakunstprojinv}
    Let $(\varphi^t)$ a smooth partially hyperbolic flow with a flow invariant subcenter foliation $\mathcal{F}^{\hat{c}}$ with $C^1$-leaves.\\
    Then there exist $\nu_0'>0$ such that 
    for every $x,y\in M$ with $d(x,y)<\nu_0'$ and $d(\varphi^1(x),\varphi^1(y))<\mu_0'$, $\pi^{wu}(x,y)$ and $\pi^{wu}\left(\varphi^1(x), \varphi^1(y)\right)$ are well defined and satisfy 
    $$\pi^{wu}\left(\varphi^1(x), \varphi^1(y)\right)=\varphi^1\left(\pi^{wu}(x,y)\right).$$
\end{lemma}
\begin{proof}
    The proof is analogous to that of Lemma \ref{lem:unstprojinv}.
\end{proof}

\begin{lemma}
\label{lem:piwu=phi^t(piu)}
    Let $(\varphi^t)$ a smooth partially hyperbolic flow with a flow invariant subcenter foliation $\mathcal{F}^{\hat{c}}$ with $C^1$-leaves.\\
    Then there exists $\nu_0''>0$ such that 
    for every $x,y\in M$ with $d(x,y)<\nu_0''$,  $\pi^{u}(x,y)$ and  $\pi^{wu}(x,y)$ are well defined and there exists a unique number $t=t(x,y)\in \left] -\nu_0'',\nu_0''\right[$ such that 
    $$\varphi^{t(x,y)}\left(\pi^{u}(x,y) \right)= \pi^{wu}(x,y).$$
\end{lemma}
\begin{proof}
    Let $x,y \in M$ such that $d(x,y)<\nu$, where $\nu<\min(\mu_0, \nu_0)$, so that $\pi^{u}(x,y)$ and  $\pi^{wu}(x,y)$ exist.
    As the proof of Proposition \ref{prop:weakunstproj} shows, if $\nu$ is sufficiently small independently of $x,y \in M$, there exists $t \in \left]-C_2^2\nu,C_2^2\nu \right [$ such that 
    $$\varphi^t(\pi^{wu}(x,y)) \in \mathcal{F}^{u}_{\mu_0}(y)\cap\bigcup_{z \in \mathcal{F}^{\hat{c}}_{\mu_0}(x)}\mathcal{F}^{ws}_{ \mu_0}(z)=\{\pi^{u}(x,y)\}.$$
    This shows the existence. We now prove the uniqueness.
    \begin{claim}
        $s:=\inf\{T>0, \text{ there exists }p\in M \text{ periodic with period }T\}>0$
    \end{claim}
    \begin{proof}[Proof of the claim]
        Otherwise, there would exist a sequence $(T_n)$ of positive numbers converging to $0$ and a sequence of points $(p_n)$ of $M$ such that for every $n \in \mathbb N$, $\varphi^{T_n}(p_n)=p_n$.
        After extraction ($M$ is compact), we can assume that $(p_n)$ converges to some $p \in M$. 
        Let $(U,\psi)$ a flow box centered at $p$. Then for some large $n$, $p_n \in U$ so $\psi(p_n)=\psi(\varphi^{T_n}(p_n))=\psi(p_n)+T_ne_1$ which is absurd since $T_n\neq 0$. 
        
    \end{proof}
    If $\nu<\frac{s}{20C_2^2}$, then such $t$ is necessarily unique, otherwise there would exist $t_1,t_2 \in \left ]-C_2^2\nu, C_2^2\nu \right [$ such that $\varphi^{t_1-t_2}(\pi^{u}(x,y))=\pi^{u}(x,y)$, and therefore $ t_1-t_2\in \left ]-2C_2^2\nu, 2C_2^2\nu \right [] $ would necessarily equal $0$.
\end{proof}

\begin{remark}
    As the proof of this last Lemma shows, $|t(x,y)|<C_2^2\nu$ if $d(x,y)<\nu$.
\end{remark}

\begin{proposition}
\label{prop:dyncohsubcen}
    Let $(\varphi^t)$ a smooth partially hyperbolic flow with a flow invariant subcenter foliation $\mathcal{F}^{\hat{c}}$ with $C^1$-leaves. Assume that the subcenter foliation $\mathcal{F}^c$ is complete.\\
    Then there exists a continuous foliation $\mathcal{F}^{\hat{c}s}$ with $C^1$ leaves tangent to $E^{\hat{c}s}$.
    Its leaf at $x \in M$ is $\mathcal{F}^{\hat{c}s}(x):=\bigcup_{z \in \mathcal{F}^{\hat{c}}(x)}\mathcal{F}^{s}(z)$.
    In particular, for every $t \in \mathbb R$ and $x \in M$, $\varphi^t(\mathcal{F}^{\hat{c}s}(x))=\mathcal{F}^{\hat{c}s}(\varphi^t(x))$.
    \end{proposition}
\begin{proof}
    We first prove that for every $x \in M$, $\mathcal{F}^{\hat{c}s}(x)$ is an immersed $C^1$ submanifold of $M$.
    Let $\gamma>0$ small enough so that $$U_\gamma:=\bigcup_{|t|<\gamma} \varphi^t \left(\bigcup_{z \in \mathcal{F}_\gamma^{\hat{c}}(x)}\mathcal{F}_\gamma^{s}(z) \right)$$
    is an open subset thus an embedded $C^1$ submanifold of $\mathcal{F}^{cs}(x)$.
    Consider, for every $y \in \mathcal{F}^{\hat{c}s}_\gamma(x):=\bigcup_{z \in \mathcal{F}_\gamma^{\hat{c}}(x)}\mathcal{F}_\gamma^{s}(z)$, a small embedded $C^1$ disk $D_y$ in $U_\gamma$ containing $y$ and tangent to $E^{\hat{c}s}$ at $y$.
    If $D_y$ is small enough, it comes by Lemma \ref{lem:piwu=phi^t(piu)} that for every $w \in D_y$, there exists a unique $t_w \in \left ] -\nu_0'', \nu_0'' \right [$ such that $\varphi^{t_w}(w)\in \bigcup_{z \in \mathcal{F}_\gamma^{\hat{c}}(y)}\mathcal{F}_\gamma^{s}(z)$, since $\pi^u(y,w)=w$ for $w\in D_y$.
    The map $\psi_y:w \in D_y \mapsto\varphi^{t_w}(w)\in \mathcal{F}^{\hat{c}s}_\gamma(y)$ is a homeomorphism onto its image which is a neighborhood of $y$ in $\mathcal{F}^{\hat{c}s}_\gamma(x)$.
    This proves that $(\psi_y(D_y), \psi_y^{-1})_{y \in \mathcal{F}^{\hat{c}s}_\gamma(x) }$ defines a $C^0$ atlas of charts for $\mathcal{F}^{\hat{c}s}_\gamma(x)$.
    Moreover, by construction, for $y,y' \in \mathcal{F}^{\hat{c}s}_\gamma(x)$ such that $\psi_y(D_y) \cap \psi_{y'}(D_{y'}) \neq \emptyset$, the transition map $\psi_{y}^{-1}\circ \psi_{y'}$ is given by holonomy maps for the restriction of $\Phi$ to $\mathcal{F}^{cs}(x)$ which is a $C^1$ foliation of $\mathcal{F}^{cs}(x)$.
    Therefore $(\psi_y(D_y), \psi_y^{-1})_{y \in \mathcal{F}^{\hat{c}s}_\gamma(x) }$ defines a $C^1$ atlas of charts for $\mathcal{F}^{\hat{c}s}_\gamma(x)$ which makes the latter a $C^1$ immersed submanifold of $\mathcal{F}^{cs}(x)$ and thus of $M$, tangent to $E^{\hat{c}s}$ by construction.
    As a result, $\mathcal{F}^{\hat{c}s}(x)$ is a $C^1$ immersed submanifold of $\mathcal{F}^{cs}(x)$ and of $M$.\\
    We now construct a chart at $x \in M$ adapted to $\mathcal{F}^{\hat{c}s}$.
    Let $\gamma,c'>0$ small enough so that the map
    $$\eta:\left \{\begin{array}{ccl}
         \mathcal{F}_{c'\gamma}^{\hat{c}s}(x)\times \mathcal{F}_{c'\gamma}^{wu}(x) & \to & M \\
         (y,z) & \to & \mathcal{F}_{\gamma}^{wu}(y)\cap \mathcal{F}_{\gamma}^{\hat{c}s}(z)
    \end{array} \right.$$
    is a well-defined homeomorphism onto its image by Proposition \ref{prop:weakunstproj}. 
    By identifying $\mathcal{F}_{c'\gamma}^{\hat{c}s}(x)$ and $\mathcal{F}_{c'\gamma}^{wu}(x)$ to open connected subsets of $\mathbb R^{d_{\hat{c}}+d_s}$ and $\mathbb R^{d_{wu}}$ respectively, we see that $\eta^{-1}$ is an adapted foliation chart for $\mathcal{F}^{\hat{c}s}$ whose restriction at plaques of $\mathcal{F}^{\hat{c}s}$ is $C^1$.
\end{proof}

\begin{lemma}
\label{lem:t_inv}
    Let $(\varphi^t)$ a smooth partially hyperbolic flow with a flow invariant subcenter foliation $\mathcal{F}^{\hat{c}}$ with $C^1$-leaves.\\
    Then there exists $\nu_0'''>0$ such that 
    for every $x,y\in M$ with $d(x,y)<\nu_0'''$ and $d(\varphi^1(x),\varphi^1(y))<\nu_0'''$, $t(x,y)$ and $t(\varphi^1(x),\varphi^1(y))$ are well defined and satisfy 
    $$t(\varphi^1(x),\varphi^1(y))=t(x,y).$$
\end{lemma}
\begin{proof}
    This is an immediate consequence of Lemma \ref{lem:unstprojinv}, Lemma \ref{lem:weakunstprojinv} and Lemma \ref{lem:piwu=phi^t(piu)}.
\end{proof}

\begin{lemma}
\label{lem:weakunstprojtrivholo}
    Let $(\varphi^t)$ a smooth partially hyperbolic flow with a flow invariant subcenter foliation $\mathcal{F}^{\hat{c}}$ with $C^1$-leaves and trivial holonomy.\\
    Then there exist $\nu_1>0$ such that 
    for every $x\in M$ and subcenter leaf $L^{\hat{c}}$ with $d(x,L)<\nu_1C'$, the set 
    $$\mathcal{F}^{wu}_{\nu_1}(x)\cap\bigcup_{z \in L^{\hat{c}}}\mathcal{F}^{s}_{ \nu_1}(z)$$
    consists of a single point.\\
    Moreover, for every $x \in M$, $\mathcal{F}^{s}_{2\nu_1}(x)\cap\mathcal{F}^{\hat{c}}(x)=\{x\}$.
\end{lemma}
\begin{proof}
    The proof is completely analogous to that of Lemma \ref{lem:unstprojtrivholo}.
\end{proof}

In particular, the previous corollary implies that for every $x\in M$, every point $y$ in $\bigcup_{z \in \mathcal{F}^{\hat{c}}(x)}\mathcal{F}^{s}_{ \nu_1}(z)$ belongs to some $\mathcal{F}^{s}_{ \nu_1}(z)$ for a unique $z\in \mathcal{F}^{\hat{c}}(x)$.

\begin{lemma}
\label{lem:Piwu}
    Let $(\varphi^t)$ a smooth partially hyperbolic flow with a flow invariant subcenter foliation $\mathcal{F}^{\hat{c}}$ with $C^1$-leaves and trivial holonomy.\\
    Then there exist $\nu_2>0$ such that 
    for every $x,y\in M$ with $d(x,y)<\nu_2$ and every $w \in \mathcal{F}^{\hat{c}}(y)$, the set 
        $$\mathcal{F}^{wu}_{\nu_1}(w)\cap\bigcup_{z \in \mathcal{F}^{\hat{c}}(x)}\mathcal{F}^{s}_{ \nu_1}(z)$$
        consists of a single element denoted by $\Pi^{wu}_{x,y}(w)$.\\
        Moreover, the map $w \mapsto \Pi^{wu}_{x,y}(w)$ is continuous.
\end{lemma}
\begin{lemma}
\label{lem:Piwuinv}
    Let $(\varphi^t)$ a smooth partially hyperbolic flow with a flow invariant subcenter foliation $\mathcal{F}^{\hat{c}}$ with $C^1$-leaves and trivial holonomy.\\
    Then there exist $\nu_3>0$ such that 
    for every $x,y\in M$ with $d(x,y)<\nu_3$ and $d(\varphi^1(x),\varphi^1(y))<\nu_3$, for every $w \in \mathcal{F}^{\hat{c}}(y)$, $\Pi^{wu}_{x,y}(w)$ and $\Pi^{wu}_{\varphi^1(x),\varphi^1(y)}(\varphi^1(w))$ are well defined and satisfy 
    $$\Pi^{wu}_{\varphi^1(x),\varphi^1(y)}(\varphi^1(w))=\varphi^1\left(\Pi^{wu}_{x,y}(w) \right).$$
\end{lemma}
\begin{proof}
    This is immediate by the previous Lemma and Lemma \ref{lem:Piuinv}.

\end{proof}

\begin{lemma}
\label{lem:Piwu=phi^t(Piu)}
    Let $(\varphi^t)$ a smooth partially hyperbolic flow with a flow invariant subcenter foliation $\mathcal{F}^{\hat{c}}$ with $C^1$-leaves and trivial holonomy.\\
    Then there exist $\nu_4>0$ such that 
    for every $x,y\in M$ with $d(x,y)<\nu_4$ and every $w \in \mathcal{F}^{\hat{c}}(y)$, $\Pi^{u}_{x,y}(w)$ and $\Pi^{wu}_{x,y}(w)$ are well-defined and there exists a unique number $t(x,y,w)\in \left ]-\nu_4, \nu_4 \right[$ such that 
    $$\Pi^{wu}_{x,y}(w)=\varphi^{t(x,y,w)}\left(\Pi^{u}_{x,y}(w) \right).$$
\end{lemma}
\begin{proof}
    This is immediate by Lemma \ref{lem:piwu=phi^t(piu)}.

\end{proof}
\begin{lemma}
\label{lem:tinv}
    Let $(\varphi^t)$ a smooth partially hyperbolic flow with a flow invariant subcenter foliation $\mathcal{F}^{\hat{c}}$ with $C^1$-leaves and trivial holonomy.\\
    Then there exist $\nu_4'>0$ such that 
    for every $x,y\in M$ with $d(x,y)<\nu_4'$ and $d(\varphi^1(x),\varphi^1(y))<\nu_4'$, for every $w \in \mathcal{F}^{\hat{c}}(y)$, $t(x,y,w)$ and $t(\varphi^1(x),\varphi^1(y),\varphi^1(w))$ are well defined and satisfy 
    $$t(\varphi^1(x),\varphi^1(y),\varphi^1(w))=t(x,y,w).$$
\end{lemma}
\begin{proof}
    This is immediate by the previous Lemma, Lemma \ref{lem:Piuinv} and Lemma \ref{lem:Piwuinv}.
\end{proof}

\subsection{Completeness of $E^{\hat{c}}\oplus E^s$}
Let $x \in M$, $y \in \mathcal{F}^{s}_{ \nu_4'}(x)$ and $w \in \mathcal{F}^{\hat{c}}(y)$.
We are going to prove that $\Pi^{wu}_{x,y}(w)=w$ which will imply that $w\in \bigcup_{z \in \mathcal{F}^{\hat{c}}(x)}\mathcal{F}^{s}_{ \nu_1}(z)$.
By Corollary \ref{cor:compcent}, we already know that $\Pi^{u}_{x,y}(w)=w$ and by Lemma \ref{lem:Piwu=phi^t(Piu)} it comes $\Pi^{wu}_{x,y}(w)=\varphi^{t(x,y,w)}(w)$.
By the choice of $\nu$, it is equivalent to prove that $t(x,y,w)=0$.
\begin{lemma}
\label{lem:t=0}
    Let $(\varphi^t)$ a smooth partially hyperbolic flow with a flow invariant subcenter foliation $\mathcal{F}^{\hat{c}}$ with $C^1$-leaves and trivial holonomy. \\
    Then there exist $\nu_4''>0$ such that 
    for every $x\in M$, $y \in \mathcal{F}^{s}_{ \nu_4''}(x)$ and $w \in \mathcal{F}^{\hat{c}}(y)$, $t(x,y,w)$ is well-defined and equals $0$.
\end{lemma}
\begin{proof}
Let $x \in M$, $y \in \mathcal{F}^s_\nu(x)$ and $w \in \mathcal{F}^{\hat{c}}(y)$, where $\nu< \nu_4'$, so that $t(x,y,w)$ is well-defined by \ref{lem:Piwu=phi^t(Piu)}.
For every $n \in \mathbb N$, it comes that $\varphi^n(y) \in \mathcal{F}^{s}_{\nu}(\varphi^n(x))$ and $\varphi^n(w)\in \mathcal{F}^{\hat{c}}(\varphi^n(y))$. 
By the previous Lemma, for every $n \in \mathbb N$:
$$t(\varphi^n(x),\varphi^n(y),\varphi^n(w))=t(x,y,w)=:t.$$
Let $z\in \mathcal{F}^{\hat{c}}(x)$ so that $\Pi^{wu}_{x,y}(w)=\varphi^t(w)\in \mathcal{F}^s_{\nu_1}(z)$.
As $d(\varphi^n(x),\varphi^n(y))\xrightarrow[n \to +\infty]{}0$ and $d(\varphi^n(\varphi^t(w)),\varphi^n(z))\xrightarrow[n \to +\infty]{}0$, the Hausdorff distance between $\mathcal{F}^{\hat{c}}(\varphi^n(w))=\mathcal{F}^{\hat{c}}(\varphi^n(y))$ and $\mathcal{F}^{\hat{c}}(\varphi^n(x))$, as well as the Hausdorff distance between $\mathcal{F}^{\hat{c}}(\varphi^n(\varphi^t(w)))=\varphi^t(\mathcal{F}^{\hat{c}}(\varphi^n(w)))$ and $\mathcal{F}^{\hat{c}}(\varphi^n(z))=\mathcal{F}^{\hat{c}}(\varphi^n(x))$, converge to $0$ as $n \to + \infty$. 
Therefore, the Hausdorff distance between $\mathcal{F}^{\hat{c}}(\varphi^n(w))$ and $\varphi^t(\mathcal{F}^{\hat{c}}(\varphi^n(w)))$ converges to $0$ as $n \to + \infty$.
As the leaf space $M/\mathcal{F}^{\hat{c}}$ is a compact metric space with the Hausdorff metric, after extraction, there exists $w_*\in M$ such that $\mathcal{F}^{\hat{c}}(\varphi^t(w^*))=\mathcal{F}^{\hat{c}}(w^*)$.
If $\nu$ is small enough independently of $x,y,w$, then $\varphi^t(w^*)\in \mathcal{F}^{\hat{c}}(w^*) \cap \Phi_{\mu_*}(w^*)=\{w^*\}$ (by Lemma \ref{lem:transvcompafol}) i.e. $\varphi^t(w^*)=w^*$.
By the construction of $t$, necessarily $t=0$.
\end{proof}
\begin{proposition}
\label{prop:Fçsloc=FçxFsloc}
    For every $x\in M$, the map
    $$\left \{\begin{array}{ccl}
        \mathcal{F}^s_{\nu_4''}(x) \times \mathcal{F}^{\hat{c}}(x)&\to & M\\
         (y,z) &\mapsto & \mathcal{F}^{\hat{c}}(y)\cap \mathcal{F}^s_{\nu_1}(z)
    \end{array} \right.$$
    is a well-defined homeomorphism onto its image.
\end{proposition}

\begin{proof}
    Let $y \in \mathcal{F}^s_{\nu_4''}(x)$ and $z \in \mathcal{F}^{\hat{c}}(x)$. Then $x \in \mathcal{F}^s_{\nu_4''}(y)$ so by the previous result, there exists a point $w \in \mathcal{F}^{\hat{c}}(y)$ such that $z\in \mathcal{F}^s_{\nu_1}(w)$ which is unique. Therefore, the above map is well-defined and it can be checked that it is a homeomorphism onto its image.
    
\end{proof}

\begin{corollary}
\label{cor:Fçs->Fçsubm}
    For every $x \in M$,
    $$\bigcup_{z \in \mathcal{F}^{s}(x)}\mathcal{F}^{\hat{c}}(z) \subset \bigcup_{z \in \mathcal{F}^{\hat{c}}(x)}\mathcal{F}^{s}(z).$$
\end{corollary}
\begin{proof}
    Let $z \in \mathcal{F}^s(x)$ and $w \in \mathcal{F}^{\hat{c}}(z)$. 
    For $n\in \mathbb N$ large enough, $\varphi^n(z) \in \mathcal{F}^s_{\nu_4''}(\varphi^n(x))$ so by the previous Proposition $t(\varphi^n(x),\varphi^n(z),\varphi^n(w))=0$, i.e. 
    $\Pi^{wu}_{\varphi^n(x),\varphi^n(z)}(\varphi^n(w))=\varphi^n(w)$ which implies that 
    $$\varphi^n(w) \in \bigcup_{z' \in \mathcal{F}^{\hat{c}}(\varphi^n(x))}\mathcal{F}^{s}(z')=\varphi^n\left(\bigcup_{z' \in \mathcal{F}^{\hat{c}}(x)}\mathcal{F}^{s}(z')\right) $$
    and therefore $w \in \bigcup_{z' \in \mathcal{F}^{\hat{c}}(x)}\mathcal{F}^{s}(z')$ which concludes.
\end{proof}
\begin{corollary}
\label{cor:compsubc}
    For every $x \in M$,
    $$\bigcup_{z \in \mathcal{F}^{s}(x)}\mathcal{F}^{\hat{c}}(z) = \bigcup_{z \in \mathcal{F}^{\hat{c}}(x)}\mathcal{F}^{s}(z).$$
\end{corollary}
\begin{proof}
    It suffices to prove $\bigcup_{z \in \mathcal{F}^{\hat{c}}(x)}\mathcal{F}^{s}(z) \subset \bigcup_{z \in \mathcal{F}^{s}(x)}\mathcal{F}^{\hat{c}}(z)$ by the previous Proposition.
    Let $z \in \mathcal{F}^{\hat{c}}(x)$ and $y \in \mathcal{F}^s(z)$. Then $z \in \mathcal{F}^s(y)$ and $x \in \mathcal{F}^{\hat{c}}(z)$ so by the previous point $x \in \bigcup_{w \in \mathcal{F}^{{\hat{c}}}(y)}\mathcal{F}^{s}(w)$ i.e. there exists $w\in \mathcal{F}^{\hat{c}}(y)$ such that $x \in \mathcal{F}^s(w)$. 
    Therefore, $y\in \mathcal{F}^{\hat{c}}(w)$ with $w \in \mathcal{F}^s(x)$ which concludes.
\end{proof}

\begin{proposition}
\label{prop:locprodstruquotient}
    For every $x \in M$, the projection
    $$y\ \mapsto \mathcal{F}^{s}(y)\cap \mathcal{F}^{\hat{c}}(x)$$
    from $\mathcal{F}^{\hat{c}s}(x)$ to $\mathcal{F}^{\hat{c}}(x)$
    is a well-defined map, which is moreover a $C^1$ submersion if $(\varphi^t)$ is subcenter-bunched.
\end{proposition}
\begin{proof}
    Let $y \in \mathcal{F}^{\hat{c}s}(x)=\bigcup_{z \in \mathcal{F}^{\hat{c}}(x)}\mathcal{F}^{s}(z)$. 
    Let $z_1,z_2 \in \mathcal{F}^{\hat{c}}(x)$ such that for $i \in \{1,2\}$, $y \in \mathcal{F}^s(z_i)$.
    Therefore, $d(\varphi^n(z_1),\varphi^n(z_2)) \xrightarrow[n\to +\infty]{}0$ so for $n \in \mathbb N$ large enough, $\varphi^n(z_1)\in \mathcal{F}^s_{\nu_1}(\varphi^n(z_2))\cap \mathcal{F}^{\hat{c}}(\varphi^n(z_2))=\{\varphi^n(z_2)\}$ by Lemma \ref{lem:transvcompafol}. This proves $z_1=z_2$ and the projection map onto $\mathcal{F}^{\hat{c}}(x)$ is well-defined.
    The fact that it is a $C^1$ submersion if $(\varphi^t)$ is subcenter-bunched comes from Theorem \ref{prop:subunchC1} since the local stable holonomy inside a subcenter-stable leaf would be $C^1$.
\end{proof}

The proof of Theorem \ref{thm:subcentcomp} shows that $M$ has local product structure in the sense of \cite{wang_quasi-shadowing_2023}.
Therefore, we have the following result:

\begin{corollary}
    Let $(\varphi^t)$ a partially hyperbolic flow on a smooth compact manifold $M$ whose subcenter distribution is integrable to a flow invariant compact foliation $\mathcal{F}^{\hat{c}}$ with trivial holonomy.\\
    Then $(\varphi^t)$ has the quasi-shadowing property on $M$.
\end{corollary}

\section{Local product structure in the leaf space $M/\mathcal{F}^{\hat{c}}$}
\label{sec:7}

Let $(\varphi^t)$ be a partially hyperbolic flow with a flow invariant subcenter foliation $\mathcal{F}^{\hat{c}}$ with $C^1$ leaves on a smooth compact manifold $M$.
As it has been mentioned in Theorem \ref{thm:carra}, the leaf space $\overline{M}:=M/\mathcal{F}^{\hat{c}}$ is a compact topological manifold and the quotient map $p:M\to \overline{M}$ is a continuous fiber bundle whose fibers are the subcenter leaves, by Reeb's stability theorem \ref{thm:reeb}.
For a set $A\subset M$, we denote by $\overline{A}=p(A)$ its $\mathcal{F}^{\hat{c}}$-saturation. The same notation can be used for points.

As mentioned earlier, the center-unstable and subcenter-stable foliations project on $\overline{M}$ to topological foliations. We prove that they form a local product structure:
\begin{proposition}
    There exist $0<\epsilon<\epsilon'$ such that for every $x\in M$,
    the map
    $$\begin{array}{ccl}
        \overline{\mathcal{F}^{wu}_\epsilon(x)} \times \overline{\mathcal{F}^{s}_\epsilon(x)} & \to & \overline{M} \\
        (\overline{y},\overline{z}) & \mapsto& \overline{\mathcal{F}^{wu}_{\epsilon'}(z)} \cap \overline{\mathcal{F}^{s}_{\epsilon'}(y)}
    \end{array}$$
    is a well-defined homeomorphism onto its image.
\end{proposition}
We need the following lemma:

\begin{lemma}
    There exits $\nu>0$ such that for every $x,y \in M$ with $d(x,y)<\nu$, the map $w\in \mathcal{F}^{\hat{c}}(y)\mapsto \Pi^{wu}_{x,y}(w)\in \mathcal{F}^{\hat{c}}(\Pi^{wu}_{x,y}(y)) $ is a well-defined homeomorphism.
\end{lemma}
\begin{proof}
    First take $\nu<\nu_2$ so that for every $w\in \mathcal{F}^{\hat{c}}(y)$, $\Pi^{wu}_{x,y}(w)$ is well defined. Then by the same proof as that of Proposition \ref{prop:Fçsloc=FçxFsloc}, if $\nu$ is sufficiently small and $d(x,y)<\nu$, the map $w\in \mathcal{F}^{\hat{c}}(y)\mapsto \Pi^{wu}_{x,y}(w)\in \mathcal{F}^{\hat{c}}(\Pi^{wu}_{x,y}(y)) $ is well-defined and surjective. The injectivity and continuity of this map as well as its inverse can be checked.
\end{proof}
\begin{proof}[Proof of Proposition \ref{prop:locprodstruquotient}]
    By the previous result, there exist $0<\epsilon<\epsilon'$ such that for every $x \in M$, for every $y\in \mathcal{F}^{\hat{c}}(\mathcal{F}^{wu}_\epsilon(x) )$ and $z \in \mathcal{F}^{\hat{c}}(\mathcal{F}^{s}_\epsilon(x))$, 
    $\mathcal{F}^{\hat{c}}(\mathcal{F}^{wu}_{\epsilon'}(z))\cap \mathcal{F}^{\hat{c}}(\mathcal{F}^{s}_{\epsilon'}(y))$ is exactly a subcenter leaf $L_{x,y,z}$.
    If $y'\in \mathcal{F}^{\hat{c}}(y)$ and $z'\in \mathcal{F}^{\hat{c}}(z)$, then by the previous result it comes that $L_{x,y',z'}=L_{x,y,z}$.
    Therefore, the map $$\eta_x:\left \{\begin{array}{ccl}
        \overline{\mathcal{F}^{wu}_\epsilon(x)} \times \overline{\mathcal{F}^{s}_\epsilon(x)} & \to & \overline{M} \\
        (\overline{y},\overline{z}) & \mapsto& \overline{\mathcal{F}^{wu}_{\epsilon'}(z)} \cap \overline{\mathcal{F}^{s}_{\epsilon'}(y)}
    \end{array}\right .$$
    is well defined, and it can be checked that it is a homeomorphism onto its image.
\end{proof}
\bibliographystyle{abbrv}
\bibliography{article_3.bib}

@article{fang_rigidity_2007,
	title = {On the rigidity of quasiconformal {Anosov} flows},
	volume = {27},
	copyright = {https://www.cambridge.org/core/terms},
	issn = {0143-3857, 1469-4417},
	url = {https://www.cambridge.org/core/product/identifier/S0143385707000326/type/journal_article},
	doi = {10.1017/S0143385707000326},
	language = {en},
	number = {6},
	urldate = {2024-12-11},
	journal = {Ergodic Theory and Dynamical Systems},
	author = {Fang, Yong},
	month = dec,
	year = {2007},
	pages = {1773--1802},
}

@book{camacho_geometric_1985,
	address = {Boston, MA},
	series = {{SpringerLink} {Bücher}},
	title = {Geometric {Theory} of {Foliations}},
	isbn = {978-1-4612-5292-4},
	doi = {10.1007/978-1-4612-5292-4},
	language = {eng},
	publisher = {Birkhäuser},
	author = {Camacho, César and Lins Neto, Alcides},
	year = {1985},
}

@book{moerdijk_introduction_2003,
	address = {Cambridge, U.K New York},
	series = {Cambridge studies in advanced mathematics},
	title = {Introduction to foliations and {Lie} groupoids},
	isbn = {978-0-511-61545-0 978-0-511-06307-7 978-0-511-07153-9 978-0-521-83197-0},
	language = {eng},
	number = {91},
	publisher = {Cambridge University Press},
	editor = {Moerdijk, Ieke and Mrčun, J.},
	year = {2003},
}

@book{lee_manifolds_2009,
	address = {Providence, Rhode Island},
	series = {Graduate {Studies} in {Mathematics}},
	title = {Manifolds and differential geometry},
	isbn = {978-1-4704-1170-1},
	language = {eng},
	number = {Volume 107},
	publisher = {American Mathematical Society},
	author = {Lee, Jeffrey M.},
	year = {2009},
}

@book{barreira_nonuniform_2007,
	address = {Cambridge},
	series = {Encyclopedia of mathematics and its applications},
	title = {Nonuniform hyperbolicity: dynamics of systems with nonzero {Lyapunov} exponents},
	isbn = {978-0-521-83258-8},
	shorttitle = {Nonuniform hyperbolicity},
	language = {eng},
	number = {115},
	publisher = {Cambridge university press},
	author = {Barreira, Luís Manuel and Pesin, Yakov B.},
	year = {2007},
}

@book{hirsch_invariant_1977,
	address = {Berlin New York},
	series = {Lecture notes in mathematics},
	title = {Invariant manifolds},
	isbn = {978-3-540-37382-7},
	language = {eng},
	number = {583},
	publisher = {Springer},
	author = {Hirsch, Morris W. and Pugh, Charles C. and Shub, Michael},
	year = {1977},
}

@article{wang_quasi-shadowing_2023,
	title = {Quasi-{Shadowing} for {Partially} {Hyperbolic} {Flows} with a {Local} {Product} {Structure}},
	volume = {29},
	copyright = {2021 The Author(s), under exclusive licence to Springer Science+Business Media, LLC, part of Springer Nature},
	issn = {1573-8698},
	url = {https://link.springer.com/article/10.1007/s10883-021-09573-y},
	doi = {10.1007/s10883-021-09573-y},
	language = {en},
	number = {1},
	urldate = {2025-12-30},
	journal = {Journal of Dynamical and Control Systems},
	publisher = {Springer US},
	author = {Wang, Lin},
	month = jan,
	year = {2023},
	note = {Company: Springer
Distributor: Springer
Institution: Springer
Label: Springer},
	pages = {95--110},
}

@misc{bohnet_partially_2014,
	title = {Partially hyperbolic diffeomorphisms with uniformly center foliation: the quotient dynamics},
	shorttitle = {Partially hyperbolic diffeomorphisms with uniformly center foliation},
	url = {http://arxiv.org/abs/1210.2835},
	doi = {10.48550/arXiv.1210.2835},
	urldate = {2025-12-30},
	publisher = {arXiv},
	author = {Bohnet, Doris and Bonatti, Christian},
	month = dec,
	year = {2014},
	note = {arXiv:1210.2835 [math]},
}

@phdthesis{bohnet_partially_2011,
	type = {Theses},
	title = {Partially hyperbolic diffeomorphisms with a compact center foliation with finite holonomy},
	url = {https://theses.hal.science/tel-00782664},
	urldate = {2025-12-30},
	school = {Universität Hamburg},
	author = {Bohnet, Doris},
	month = dec,
	year = {2011},
}

@article{carrasco_compact_2015,
	title = {Compact dynamical foliations},
	volume = {35},
	copyright = {https://www.cambridge.org/core/terms},
	issn = {0143-3857, 1469-4417},
	url = {https://www.cambridge.org/core/product/identifier/S014338571400042X/type/journal_article},
	doi = {10.1017/etds.2014.42},
	language = {en},
	number = {8},
	urldate = {2025-12-30},
	journal = {Ergodic Theory and Dynamical Systems},
	author = {Carrasco, Pablo D.},
	month = dec,
	year = {2015},
	pages = {2474--2498},
}

@article{wilkinson_dynamical_2008,
	title = {Dynamical coherence and center bunching},
	volume = {22},
	issn = {1078-0947},
	url = {http://www.aimsciences.org/journals/displayArticles.jsp?paperID=3389},
	doi = {10.3934/dcds.2008.22.89},
	language = {en},
	number = {1/2, September},
	urldate = {2025-12-30},
	journal = {Discrete and Continuous Dynamical Systems},
	author = {Wilkinson, Amie and Burns, Keith},
	month = jun,
	year = {2008},
	pages = {89--100},
}

@article{avila_absolute_2022,
	title = {Absolute continuity, {Lyapunov} exponents, and rigidity {II}: systems with compact center leaves},
	volume = {42},
	issn = {0143-3857, 1469-4417},
	shorttitle = {Absolute continuity, {Lyapunov} exponents, and rigidity {II}},
	url = {https://www.cambridge.org/core/product/identifier/S0143385721000420/type/journal_article},
	doi = {10.1017/etds.2021.42},
	language = {en},
	number = {2},
	urldate = {2025-12-30},
	journal = {Ergodic Theory and Dynamical Systems},
	author = {Avila, A. and Viana, Marcelo and Wilkinson, A.},
	month = feb,
	year = {2022},
	pages = {437--490},
}

@article{pugh_holder_1997,
	title = {Hölder foliations},
	volume = {86},
	issn = {0012-7094},
	url = {https://projecteuclid.org/journals/duke-mathematical-journal/volume-86/issue-3/H%c3%b6lder-foliations/10.1215/S0012-7094-97-08616-6.full},
	doi = {10.1215/S0012-7094-97-08616-6},
	language = {en},
	number = {3},
	urldate = {2025-12-30},
	journal = {Duke Mathematical Journal},
	author = {Pugh, Charles and Shub, Michael and Wilkinson, Amie},
	month = jan,
	year = {1997},
}

@article{wilkinson_stable_1998,
	title = {Stable ergodicity of the time-one map of a geodesic flow},
	volume = {18},
	copyright = {https://www.cambridge.org/core/terms},
	issn = {0143-3857, 1469-4417},
	url = {https://www.cambridge.org/core/product/identifier/S0143385798117984/type/journal_article},
	doi = {10.1017/S0143385798117984},
	language = {en},
	number = {6},
	urldate = {2025-12-30},
	journal = {Ergodic Theory and Dynamical Systems},
	author = {Wilkinson, Amie},
	month = dec,
	year = {1998},
	pages = {1545--1587},
}

@book{candel_foliations_2000,
	address = {Providence, RI},
	series = {Graduate {Studies} in {Mathematics}},
	title = {Foliations {I}},
	isbn = {978-0-8218-0809-2 978-1-4704-2078-9},
	doi = {10.1090/gsm/023},
	language = {eng},
	number = {volume 23},
	publisher = {American Mathematical Society},
	author = {Candel, Alberto and Conlon, Lawrence},
	year = {2000},
}

@article{brin_partially_1974,
	title = {{PARTIALLY} {HYPERBOLIC} {DYNAMICAL} {SYSTEMS}},
	volume = {8},
	issn = {0025-5726},
	url = {https://iopscience.iop.org/article/10.1070/IM1974v008n01ABEH002101},
	doi = {10.1070/IM1974v008n01ABEH002101},
	language = {en},
	number = {1},
	urldate = {2026-01-22},
	journal = {Mathematics of the USSR-Izvestiya},
	publisher = {IOP Publishing},
	author = {Brin, M. I. and Pesin, Ja B.},
	month = feb,
	year = {1974},
	pages = {177},
}

@article{abouanass_global_2025,
	title = {On global rigidity of transversely holomorphic {Anosov} flows},
	issn = {1793-5253},
	url = {https://www.worldscientific.com/doi/10.1142/S1793525326500111},
	doi = {10.1142/S1793525326500111},
	urldate = {2026-01-22},
	journal = {Journal of Topology and Analysis},
	publisher = {World Scientific Publishing Co.},
	author = {Abouanass, Mounib},
	month = dec,
	year = {2025},
}

@article{carneiro_partially_2014,
	title = {Partially hyperbolic geodesic flows},
	volume = {31},
	issn = {0294-1449},
	url = {https://ems.press/journals/aihpc/articles/4076510},
	doi = {10.1016/j.anihpc.2013.07.009},
	language = {en},
	number = {5},
	urldate = {2026-01-22},
	journal = {Annales de l'Institut Henri Poincaré C},
	author = {Carneiro, Fernando and Pujals, Enrique},
	month = oct,
	year = {2014},
	pages = {985--1014},
}

@article{epstein_foliations_1976,
	title = {Foliations with all leaves compact},
	volume = {26},
	issn = {1777-5310},
	url = {https://www.numdam.org/item/?id=AIF_1976__26_1_265_0},
	doi = {10.5802/aif.607},
	language = {fr},
	number = {1},
	urldate = {2026-01-23},
	journal = {Annales de l'Institut Fourier},
	author = {Epstein, D. B. A.},
	year = {1976},
	pages = {265--282},
}
\end{document}